\documentclass[final]{siamltex}
\usepackage{amsfonts}
\usepackage{amssymb,amsmath,url}
\usepackage{hyperref}
\usepackage{graphicx,color,float,epstopdf}
\usepackage{color}
\usepackage{pst-blur,pstricks-add}
\numberwithin{table}{section}

\newcommand{\nn}{\nonumber}

\newtheorem{remark}{Remark}

\begin{document}
\title{\large On the Optimal Linear Convergence Rate of a Generalized Proximal Point Algorithm}

\author{Min Tao$^{1}$  \and  Xiaoming Yuan$^{2}$
\footnotetext[1]{({\tt taom@nju.edu.cn}) Department of
    Mathematics, Nanjing University, Nanjing, 210093, China. This author was supported by the Natural Science Foundation of China: NSFC-11301280 and     NSFC-11471156; and the Fundamental Research Funds for the Central Universities: 020314330019.} \footnotetext[2]{({\tt xmyuan@hkbu.edu.hk}) Department of Mathematics, Hong Kong Baptist University, Hong Kong, China. This author was supported by the General Research Fund from Hong Kong Research Grants Council: 12302514.}}

\maketitle

\begin{center}
First version: February 17, 2015; Second version: \today
\end{center}

\bigskip

\begin{abstract}
The proximal point algorithm (PPA) has been well studied in the literature. In particular, its linear convergence rate has been studied by Rockafellar in 1976 under certain condition. We consider a generalized PPA in the generic setting of finding a zero point of a maximal monotone operator, and show that the condition proposed by Rockafellar can also sufficiently ensure the linear convergence rate for this generalized PPA. Indeed we show that these linear convergence rates are optimal. Both the exact and inexact versions of this generalized PPA are discussed. The motivation to consider this generalized PPA is that it includes as special cases the relaxed versions of some splitting methods that are originated from PPA. Thus, linear convergence results of this generalized PPA can be used to better understand the convergence of some widely used algorithms in the literature. We focus on the particular convex minimization context and specify Rockafellar's condition to see how to ensure the linear convergence rate for some efficient numerical schemes, including the classical augmented Lagrangian method proposed by Hensen and Powell in 1969 and its relaxed version, the original alternating direction method of multipliers (ADMM) by Glowinski and Marrocco in 1975 and its relaxed version (i.e., the generalized ADMM by Eckstein and Bertsekas in 1992). Some refined conditions weaker than existing ones are proposed in these particular contexts.
\end{abstract}


\pagestyle{myheadings} \thispagestyle{plain} \markboth{ }{}

\section{Introduction}
\label{intr}
\setcounter{equation}{0}
Let $H$ be a real Hilbert space with inner product $\langle\cdot, \cdot\rangle$. A set-valued mapping $T: H\rightarrow2^ H$ is said to be monotone if
$$ \langle z-z', w-w' \rangle \ge 0,\;\;\forall z,\;z' \in H, \; w\in T(z), w'\in T(z').$$
$T$ is said to be maximal monotone if, in addition, its graph is not properly contained in the graph of any other monotone operator.
A fundamental problem is finding a zero point, denoted by $z^*$, of a maximal monotone set-valued mapping $T$:
\begin{eqnarray}
\label{maxsol}
0\in T(z).
\end{eqnarray}
Throughout, the set of $T$'s zero point, denoted by $zer(T)$, is assumed to be nonempty.

The proximal point algorithm (PPA), which traces back to \cite{Mar70, Moreau}, has been playing an important role both theoretically and algorithmically for (\ref{maxsol}). Starting from an arbitrary point $z^0$ in $H$, the PPA iteratively generates its sequence $\{z^k\}$ by the scheme
\begin{eqnarray}
\label{PPA}
0 \in c_kT(z^{k+1})+ z^{k+1}-z^k,
 \end{eqnarray}
where $\{c_k\}$, called proximal parameter, is a sequence of positive real numbers. Indeed, as shown in \cite{Rock76}, the convergence of PPA can be ensured when $\{c_k\}$ is bounded away from zero. Moreover, an inexact version of PPA was proposed in \cite{Rock76}, allowing the subproblem (\ref{PPA}) to be solved approximately subject to some inexactness criteria. Conceptually, the inexact version of PPA can be written as
\begin{eqnarray}
\label{PPA-inexact}
0 \approx c_kT(z^{k+1})+ z^{k+1}-z^k,
 \end{eqnarray}
in which the accuracy should be judiciously chosen to guarantee its convergence. Let
\begin{eqnarray}
\label{resolvent}
J_{c_k T}:=(I+c_k T)^{-1}
\end{eqnarray}
denote the resolvent operator of the maximal monotone set-valued mapping $T$ for a positive scalar $c_k$ (Note that $J_{c_k T}$ is single-valued, see, e.g., \cite{EcBe}). Then, the exact and inexact versions of the PPA can be written, respectively, as
\begin{eqnarray}
\label{PPAscheme}
z^{k+1} = J_{c_k T} (z^k)
 \end{eqnarray}
and
 \begin{eqnarray}
\label{PPAscheme-inexact}
z^{k+1} \approx J_{c_k T} (z^k).
 \end{eqnarray}
Technically, (\ref{PPAscheme-inexact}) includes (\ref{PPAscheme}) as the special case where the tolerance of accuracy is zero. But we still discuss them individually because (\ref{PPAscheme}) is of particular interest and it may have stronger convergence, because it requires estimating the resolvent operator accurately.

Research results on the convergence of PPA can be found in earlier literature. For example, when $T$ is specified as the sum of a single-valued,
monotone and hemicontinuous mapping and the normal cone to a bounded set, i.e., the problem (\ref{maxsol}) reduces to a variational inequality, then some convergence of the exact version of PPA (\ref{PPA}) with $c_k \equiv c$ in the weak topology was investigated in \cite{Mar70,Mar72}. In \cite{Rock76}, the convergence of both the exact and inexact versions of PPA was comprehensively studied; it is indeed the work \cite{Rock76} that popularized PPA in optimization community. More specifically, under the condition that $\{c_k\}$ is bounded away from zero, the convergence of (\ref{PPAscheme-inexact}) (thus also (\ref{PPAscheme})) in the weak topology was proved when the accuracy for ``$\approx$" in (\ref{PPAscheme-inexact}) is specified into certain forms (see (A) and (B) of Section 1 in \cite{Rock76}). In fact, the exact version (\ref{PPAscheme}) was shown to find a solution point of (\ref{maxsol}) after finitely many iterations in \cite{Rock76}. Note that the strong monotonicity of $T$ is not required for the analysis in \cite{Rock76}. Moreover, if the inverse of $T$ (denoted by $T^{-1}$) is Lipschitz continuous at $0$ (see Definition \ref{invLipze} in Section \ref{Pre} or Section 3 in \cite{Rock76}) and $\{c_k\}$ is bounded away from zero, it was proved in \cite{Rock76} that the (\ref{PPAscheme-inexact}) (thus also (\ref{PPAscheme})) with some relative error control in its accuracy is linearly convergent. There are many other articles studying the PPA from various perspectives. For example, the PPA application to nonconvex problems studied in \cite{Fuk}, the PPA with variable metric in \cite{BQ}, a unified convergence rate analysis for some PPA-based decomposition methods in \cite{ST}, accelerated PPA schemes with a worst-case $O(1/k^2)$ convergence rate proposed in \cite{Guler92}, the logarithmic quadratic proximal extension considered in \cite{AusTeb05,ATB99}, and some other proximal-type algorithms in \cite{Teb}. We particularly refer to \cite{Guler91,Nemirovski2005} for some insightful analysis on the iteration complexity of PPA, which can be regarded as a measure of its worst-case convergence rate. Algorithmically, the PPA is the basis of a large number of celebrated methods, e.g., the projected gradient method \cite{Polyak}, the extragradient method \cite{Kor}, the extended extragradient and hyperplane schemes in \cite{AusTeb05}, the forward-backward operator splitting method \cite{LM}, and the accelerated projected gradient method \cite{Nesterov}.

%{\bf
%$$
%[Assumption]: \quad\hbox{$T^{-1}$ is Lipschitz continuous at $0$},
%$$
%}

As studied in \cite{CYuan,EcBe,Gabay83,Gol79}, the PPA schemes (\ref{PPAscheme}) and (\ref{PPAscheme-inexact}) can be generalized, respectively, as
\begin{eqnarray}\label{RPPAscheme-ck}
[\hbox{Exact Version}]: \quad \quad z^{k+1}= z^k - \gamma (z^k - J_{c_k T}(z^k)),
\end{eqnarray}
and
\begin{eqnarray}\label{RPPAscheme-ck-inexact}
[\hbox{Inexact Version}]: \quad \quad z^{k+1} \approx z^k - \gamma (z^k - J_{c_k T}(z^k)).
\end{eqnarray}
In (\ref{RPPAscheme-ck}) and (\ref{RPPAscheme-ck-inexact}), the proximal parameter sequence $\{c_k\}$ is also required to be bound away from zero, i.e., $c_k\ge \kappa>0$ for all $k$, and the relaxation factor $\gamma \in (0,2)$. The generalized PPA schemes (\ref{RPPAscheme-ck}) and (\ref{RPPAscheme-ck-inexact}) usually can accelerate the original PPA schemes numerically, see, e.g., \cite{Ber82, CGHY,FLHY} for some numerical verifications. Thus, from the PPA perspective itself, it is interesting to consider its generalized versions. Another reason of considering the generalized PPA schemes (\ref{RPPAscheme-ck}) and (\ref{RPPAscheme-ck-inexact}) is that the original PPA scheme (\ref{PPAscheme}) indeed is a unified illustration of some different schemes for different models ---- it has been well studied that some popular iterative schemes such as the Douglas-Rachford splitting method (DRSM) in \cite{DR56,LM}, the Peaceman-Rachford splitting method in \cite{PR,LM} and the augmented Lagrangian method (ALM) in \cite{Hes,Powell} are all special cases of the PPA (\ref{PPAscheme}) with specific choices of $T$. Thus, generalizing the PPA scheme (\ref{PPAscheme}) (Resp.,(\ref{PPAscheme-inexact})) as (\ref{RPPAscheme-ck}) (Resp., (\ref{RPPAscheme-ck-inexact})) represents a unified consideration for accelerating a series of well known splitting algorithms, especially in the convex optimization context. Let us just elaborate on the detail of the DRSM. Recall that (see \cite{EcBe}, also Section \ref{ADMMChar} for details) the DRSM is a special case of the PPA (\ref{PPAscheme}). In \cite{Gabay83}, it was proved that the alternating direction method of multipliers (ADMM), which was originally proposed in \cite{GM} and now finds many applications in a wide range of areas, is a special case of the DRSM. Thus, the ADMM is also a special case of (\ref{PPAscheme}) and it can be accelerated immediately by the scheme (\ref{RPPAscheme-ck}). This application inspired the so-called generalized ADMM in \cite{EcBe}, whose acceleration effectiveness was demonstrated recently in \cite{FLHY} by some statistical learning applications.

Our main purpose is to extend the analysis in \cite{Rock76} to the generalized PPA schemes (\ref{RPPAscheme-ck}) and (\ref{RPPAscheme-ck-inexact}), and establish their linear convergence rates under the same assumption as \cite{Rock76}: {\bf $T^{-1}$ is Lipschitz continuous at $0$}. We further show that these linear convergence rates are indeed optimal. Because of the just-mentioned explanation, studying the linear convergence of the generalized schemes (\ref{RPPAscheme-ck}) and (\ref{RPPAscheme-ck-inexact}) helps us better understand the convergence properties of a number of specific splitting methods in the convex optimization context through a unified perspective. In \cite{Gabay83}, the linear convergence of the exact version (\ref{RPPAscheme-ck}) with $c_k\equiv c$ and $\gamma\in(1,2)$ was discussed under the assumptions that $T$ is both strongly monotone (see Definition \ref{coer}) and Lipschitz continuous. In \cite{CYuan}, also under the assumption that $T$ is strongly monotone, the sublinear and linear convergence rates of the schemes (\ref{RPPAscheme-ck}) and (\ref{RPPAscheme-ck-inexact}) with $c_k\equiv c$ was studied; and these results were especially specified for the DRSM and PRSM scenarios. The results in \cite{CYuan} were then refined in \cite{DYin2} for the special DRSM and PRSM cases of (\ref{PPA}). Note that, as analyzed in \cite{Rock76}, ``the assumption of Lipschitz continuity of $T^{-1}$ at 0 turns out to be very natural in applications to convex programming". Indeed, we will show later that this assumption is weaker than those considered in \cite{CYuan,DYin2,Gabay83} (see the example in Section \ref{Sec-Pre-Example}) and it suffices to ensure the linear convergence of the schemes (\ref{RPPAscheme-ck}) and (\ref{RPPAscheme-ck-inexact}) for the case $\gamma \in (0,2)$. Thus, the distinction of this work from existing results in the literature is that stronger convergence rates are established under weaker conditions for the generalized PPA schemes (\ref{RPPAscheme-ck}) and (\ref{RPPAscheme-ck-inexact}). We will also consider several specific convex optimization contexts of the abstract model (\ref{maxsol}) and investigate how this assumption can be specified in these special contexts to ensure the linear convergence rates for some well-studied benchmark algorithms in the literature.

The rest of this paper is organized as follows. In Section \ref{Pre}, some preliminaries useful for further analysis are summarized. Then, we discuss the convergence and the linear convergence rate of the exact version of the generalized PPA (\ref{RPPAscheme-ck}) in Section \ref{linearCon}. In Section \ref{inexPPA}, the convergence and linear convergence rate of its inexact version (\ref{RPPAscheme-ck-inexact}), in which the inexactness criterion is also specified, is studied. Then, we revisit the assumption ``$T^{-1}$ is Lipschitz continuous at $0$" in Section \ref{relCon} and show that it can be further relaxed. In Section \ref{super}, we discuss the possibility of deriving the superlinear convergence for the schemes (\ref{RPPAscheme-ck}) and (\ref{RPPAscheme-ck-inexact}). In Section \ref{ALM}, we apply the scheme (\ref{RPPAscheme-ck}) to a canonical convex minimization model with linear constraints and discuss the linear convergence for the resulting generalized ALM scheme. In Section \ref{ADMMChar}, we focus on the analysis for the linear convergence of the ADMM and the generalized ADMM scheme, both are special cases of the scheme (\ref{RPPAscheme-ck}). Finally, some conclusions are made in Section \ref{conclusion}.

\section{Preliminaries}
\label{Pre}
\setcounter{equation}{0}
In this section, we recall some definitions and known results for further discussions.

\subsection{Some Definitions}
We first recall some basic definitions to be used in our analysis.

\begin{definition}
\label{nonexp}
Let $T:H\rightarrow 2^H$ be set-valued and maximal monotone. Then, $T$ is said to be nonexpansive if
$\|w' - w\|\le \|z' - z\|,\; \;\;\;\forall \; z,z'\in H,\;w\in T(z),\; w'\in T(z')$.
\end{definition}

\begin{definition}
\label{firm-nonexp}
Let $T:H\rightarrow 2^H$ be set-valued and maximal monotone. Then, $T$ is said to be firmly nonexpansive if
$\|w'-w\|^2\le \langle z'-z,w'-w \rangle,\; \;\;\; \forall \; z,z'\in H,\;w\in T(z),\; w'\in T(z')$.
\end{definition}

\begin{definition}
\label{coer}
Let $T:H\rightarrow 2^H$ be set-valued and maximal monotone. Then, $T$ is called $\alpha$-strongly monotone if $\langle z-z',w-w'\rangle \ge\alpha\|z-z'\|^2, \; \;\forall \; z,z'\in H,\;w\in T(z),\; w'\in T(z')$ for $\alpha>0$.
\end{definition}

%\begin{definition}
%\label{firmnon}
%Let $T:H\rightarrow 2^H$ be set-valued and maximal monotone. Then, $T$ is called $\beta$-firmly nonexpansive if
%$\langle T(z)-T(z'), w-w' \rangle\ge \beta \|w-w'\|^2, \;\forall \; z,z'\in H,\;w\in T(z),\; w'\in T(z')$ for $\beta>0$.
%\end{definition}

\begin{definition} \label{invLipze}
Let $T$ be set-valued and be defined on $H$. Then, $T^{-1}$ is called Lipschitz continuous at 0 with modulus $a\ge0$ if there is a unique solution $ z^*$ to $0\in T(z)$ (i.e. $T^{-1}(0)=\{  z^*\}$), and for some $\tau>0$ we have
$\| z- z^*\|\le a\|w\|$ whenever $z\in T^{-1}(w)$ and $\|w\|\le \tau$.
%\begin{eqnarray}
%\label{invTLip}
%\| z- z^*\|\le a\|w\|\; \mbox{whenever}\; z\in T^{-1}(w)\; \mbox{and}\; \|w\|\le \tau.
%\end{eqnarray}
\end{definition}

Definition \ref{invLipze} is quoted from \cite{Rock76}. Based on these definitions, we have some immediate conclusions. For examples, if $T$ is nonexpansive, then it is Lipschitz continuous.
% (2) $T$ is $\frac{1}{\beta}$-Lipschitz continuous when $T$ is $\beta$-firmly nonexpansive [{\bf If you wan to remove the definition of $\beta$-firmly nonexpansive is removed, shall we keep this?}].
 Moreover, clearly, the problem (\ref{maxsol}) has a unique solution point when  $T^{-1}$ is Lipschitz continuous at 0.

\subsection{An Example}\label{Sec-Pre-Example}

Consider the problem (\ref{maxsol}), where $T:\Re^2\rightarrow \Re^2$ is defined by
\begin{eqnarray}\label{opeT}
T(x_1,x_2):=\frac{1}{a}(x_2,-x_1)\; \hbox{with} \;a>0.
\end{eqnarray}
Obviously, the operator $T$ defined in (\ref{opeT}) is maximal monotone and the solution point of (\ref{maxsol}) with (\ref{opeT}) is $z^*=(0,0)$. Moreover, it holds
 \begin{eqnarray}\label{invT1}
 \|T^{-1}(z_1)-T^{-1}(z_2)\|\le a\|z_1-z_2\|, \;\;\forall z_1, z_2\in\Re^2,
 \end{eqnarray}
 and
\begin{eqnarray}\label{fle}
\langle T(z_1)-T(z_2),z_1-z_2\rangle=0,  \;\;\forall z_1, z_2\in\Re^2.
\end{eqnarray}
Thus, $T^{-1}$ is Lipschitz continuous at $0$ with modulus $a>0$ while $T$ is not strongly monotone. Thus, this example shows that the assumption ``$T^{-1}$ is Lipschitz continuity at 0" is weaker than the strong monotonicity assumption on $T$ as assumed in \cite{CYuan,DYin2,Gabay83}. 

\subsection{Some Known Results}
Then, we summarize some known results that are relevant to our analysis.
The following lemma summarizes some well-known properties of a firmly nonexpansive operator. The proof is
straightforward and thus omitted, or see, e.g, \cite{EcBe}.
\begin{lemma}
\label{contr_pro}
We have the following facts.
\begin{itemize}
\item [i)] All firmly nonexpansive operators are nonexpansive.
\item [ii)] An operator $T$ is firmly nonexpansive if and only if $2T-I$ is nonexpansive.
\item [iii)] An operator is firmly nonexpansive if and only if it is of the form $\frac{1}{2}(C+I)$,
where $C$ is nonexpansive.
\item [iv)] An operator $T$ is firmly nonexpansive if and only if $I-T$ is firmly nonexpansive.
\end{itemize}
\end{lemma}

\medskip
In the following lemma, we show some simple conclusions for the resolvent operator of a maximal monotone operator.

\begin{lemma}
\label{proxity}
Let $T:H\rightarrow 2^H$ be set-valued and maximal monotone; $J_{cT}$ be defined in (\ref{resolvent}), and $c>0$ be a scalar. Then, we have
\begin{itemize}
\item [i)] $\langle J_{cT}(z) - J_{cT}( z'), (I-J_{cT})(z)-(I-J_{cT})(z') \rangle \ge0,\;\forall z, z' \in H$.
\item [ii)] $\|z-z'\|^2 \ge\|J_{cT}(z)-J_{cT}(z')\|^2 +
\|(I-J_{cT})(z)-(I-J_{cT})(z')\|^2,\;\forall z, z' \in H$.
\end{itemize}
\end{lemma}
\noindent
\begin{proof}
Obviously, $J_{cT}$ defined in (\ref{resolvent}) is nonexpansive, and it implies the first property immediately. The second property is trivial because of Property (i). $\hfill$
\end{proof}

Last, let us recall the representation lemma, see, e.g., \cite{EcBe}.
\begin{lemma}
\label{repl} (The representation lemma)
Let $c>0$ and let $T$ be monotone on $H$. Then every element $z$ of $H$ can be written in at most one way as
$x+cy$, where $y\in T(x)$. If $T$ is maximal, then every element $z$ of $H$ can be written in exactly one way as $x+cy$, where
$y\in T(x)$.
\end{lemma}

\section{Convergence of the Exact Version (\ref{RPPAscheme-ck})}
\label{linearCon}

In this section, we show that the generalized PPA (\ref{RPPAscheme-ck}) also converges linearly to a zero point of $T$
under the assumption ``$T^{-1}$ is Lipschitz continuous at $0$ with positive modulus", the same one as that in \cite{Rock76}.
For a lighter notation in analysis, we use the notation $\tilde z^k=J_{c_k T}(z^k)$ in the following analysis.

\subsection{Global Convergence}

\label{exactglocon}
First, we show the global convergence of (\ref{RPPAscheme-ck}). Note that we do not need the assumption ``$T^{-1}$ is Lipschitz continuous at $0$ with positive modulus" for proving the global convergence.
%[{\bf  I guess the proof may also be available in some literature, such as \cite{Gol79}; verify it.}]
%[{\bf you mean you did not find this reference or you found the reference but there is no proof? The latter.}]
The next theorem shows that the sequence $\{z^k\}$ generalized by (\ref{RPPAscheme-ck}) with $\gamma\in(0,2)$ is strictly contractive with respective to the solution set of (\ref{maxsol}), an important property ensuring its global convergence.

\begin{theorem}
\label{rel_PR global}
(Strict contraction) Let $\{z^k\}$ be the sequence generated by the exact version of the generalized PPA scheme (\ref{RPPAscheme-ck}) with $\gamma\in(0,2)$ and $\{c_k\}$ bounded away from 0; let $z^*$ be a solution point of (\ref{maxsol}). We have
\begin{eqnarray}
\label{cont_PR}
\|z^{k+1}-z^*\|^2\le \|z^k-z^*\|^2-\gamma(2-\gamma)\left\| z^k-\tilde z^k\right \|^2.
\end{eqnarray}
\end{theorem}

\begin{proof}
First, applying the property (ii) in Lemma \ref{proxity} with $z=z^k$ and $\tilde z=z^*$, we get
\begin{eqnarray}
\label{keycontr} \|z^k-z^*\|^2\ge\|\tilde z^k-z^*\|^2+\|z^k-\tilde z^k\|^2.
\end{eqnarray}
We thus have
\begin{eqnarray}
\label{keytmp}
\langle \tilde z^k-z^*, z^k-\tilde z^k\rangle \ge0,
\end{eqnarray}
%First, applying Property (i) in Lemma \ref{proxity} and notice the fact $z^*=J_{c_kT} z^*$ and $\tilde z^k=J_{c_k T}(z^k) $, we have
%\begin{eqnarray*}
%(\tilde z^k-z^*)^\top (z^k-\tilde z^k)\ge0,
%\end{eqnarray*}
and furthermore
\begin{eqnarray}
\label{tmp1PPA}
\langle z^k-z^*, z^k-\tilde z^k \rangle \ge\|z^k-\tilde z^k\|^2.
\end{eqnarray}
Consequently, we have
\begin{eqnarray*}
%\label{zz}
\|z^{k+1}-z^*\|^2&=&\|z^k-\gamma (z^k-\tilde z^k) - z^*\|^2\nn\\
                 &=&\|z^k-z^*\|^2-2\gamma \langle z^k-z^*, z^k-\tilde z^k \rangle +\gamma^2\|z^k -\tilde z^k\|^2\nn\\
                 &\le& \|z^k - z^*\|^2-\gamma(2-\gamma)\|z^k-\tilde z^k\|^2,
\end{eqnarray*}
where the inequality follows from (\ref{tmp1PPA}). Thus, the assertion (\ref{cont_PR}) is proved. $\hfill$
\end{proof}
\vspace{0.5cm}

Based on Theorem \ref{rel_PR global}, the convergence of (\ref{RPPAscheme-ck}) can be easily established.

\begin{theorem}
\label{thm-conv}
(Global convergence) Let $\{z^k\}$ be the sequence generated by the exact version of the generalized PPA scheme (\ref{RPPAscheme-ck}) with $\gamma\in(0,2)$ and $\{c_k\}$ bounded away from 0. Then it globally converges to a solution point of (\ref{maxsol}).
\end{theorem}
\begin{proof}
According to (\ref{cont_PR}), the sequence $\{z^k\}$ is bounded, and it has at least one accumulation point, say $z^{\infty}$. Let $\{z^{k_j}\}$ be the subsequence converging to $z^{\infty}$. Recall the notation $\tilde z^k=J_{c_k T}(z^k)$ and the definition of $J_{c_k T}$ in (\ref{resolvent}). We thus have
%From the inequality (\ref{cont_PR}), we know that the sequence $\{z^k\}$ is strictly contractive and thus it indeed converges to $z^*$.
$$c_{k_j}^{-1}(z^{k_j}-{\tilde z}^{k_j})\in T({\tilde z}^{k_j}).$$
Then, using the monotonicity of $T$, for an integer $k_j$, it holds that
\begin{eqnarray}\label{insert1}
\langle z-{\tilde z}^{k_j}, w- c_{k_j}^{-1}(z^{k_j}-{\tilde z}^{k_j})\rangle\ge0,\;\; \hbox{for all} \; z, w \;\hbox{satisfying}\;\; w \in T(z).
\end{eqnarray}
Again, it follows from (\ref{cont_PR}) that $\lim_{j\rightarrow \infty}\|z^{k_j}-\tilde z^{k_j}\|=0$. Combining this fact with $\lim_{j\rightarrow \infty}\|z^{k_j}-z^{\infty}\|=0$,
we get
$\lim_{j\rightarrow \infty}\|{\tilde z}^{k_j}-z^{\infty}\|=0$.
Recall $\{c_k\}$ is bounded away from 0. Then, taking $j\rightarrow\infty$ in (\ref{insert1}), we obtain
$$\langle z-z^{\infty}, w \rangle \ge0 \;\; \hbox{for all} \; z, w \;\hbox{satisfying}\;\; w \in T(z).
$$
In view of the maximality of $T$, this inequality implies that $z^{\infty}$ is a solution point of (\ref{maxsol}), see, e.g. \cite{Rock76}. It is easy to see from Theorem \ref{rel_PR global} that the sequence $\{z^k\}$ cannot have more than one accumulation point. Thus, $\{z^k\}$ converges to $z^{\infty}$ which a solution point of (\ref{maxsol}). The proof is complete. $\hfill$
\end{proof}

\subsection{Linear Convergence}\label{linear-exact}
\label{exactlincon}
Now, under the assumption ``$T^{-1}$ is Lipschitz continuous at $0$ with positive modulus", we prove the linear convergence of (\ref{RPPAscheme-ck}). First, two lemmas are presented.

\begin{lemma}
\label{strongm}
Let $T: H \to 2^H$ be maximal monotone and $z^*$ be a solution point of (\ref{maxsol}); let $c_k>0$. If $T^{-1}$ is Lipschitz continuous at $0$ with modulus $a>0$, then there exists a positive $\tau$ such that
\begin{equation}\label{new-1}
\|J_{c_kT}(z)-z^*\|\le \frac{a}{\sqrt{a^2+c_k^2}}\|z-z^*\|\;\; \mbox{when}\;\|c_k^{-1}(z-J_{c_kT}(z))\|\le \tau,\;\forall z\in H.
\end{equation}
\end{lemma}
\begin{proof}
Applying Property (ii) in Lemma \ref{proxity} with $\tilde z=z^*$ and $c=c_k$, we get
\begin{eqnarray}
\label{tmpLem}
\|z-z^*\|^2\ge\|J_{c_kT}(z)-z^*\|^2+\|(I-J_{c_kT})(z)\|^2.
\end{eqnarray}
Recall the definition of $J_{c_k T}$ in (\ref{resolvent}). We have
$$
c_k^{-1}(I-J_{c_kT})(z)\in T(J_{c_kT}(z)).
$$
Since $T^{-1}$ is Lipschitz continuous at $0$ with modulus $a>0$, it follows from Definition \ref{invLipze} that there exists a positive parameter
$\tau$ such that
$$\|J_{c_kT}(z)-z^*\|\le a \left\|c_k^{-1}(I-J_{c_kT})(z)\right\|\;\;\mbox{when}\; \|c_k^{-1} (I-J_{c_kT})(z)\|\le \tau.$$
Substituting this inequality into (\ref{tmpLem}), we obtain (\ref{new-1}). The proof is complete. $\hfill$
\end{proof}
\vspace{0.5cm}
\medskip

\begin{remark}
If some stronger assumptions such as ``$T$ is $\frac{1}{a}$-strongly monotone" hold as some existing work \cite{CYuan,Gabay83}, the assertion (\ref{new-1}) can be easily improved as
\begin{equation}\label{new-1s}
\|J_{c_kT}(z)-z^*\|\le \frac{a}{a+c_k}\|z-z^*\|\;\;\forall z\in H.
\end{equation}
 Under the weaker assumption ``$T^{-1}$ is Lipschitz continuous at $0$ with positive modulus", however, the assertion (\ref{new-1}) is optimal in the sense that the coefficient in the right-hand side cannot be smaller. To see this, let us consider the example (\ref{opeT}) again in Section \ref{Sec-Pre-Example}. It follows from (\ref{fle}) that
\begin{eqnarray}\label{flesol}
\langle c_k T(J_{c_kT}(z))-c_k T(J_{c_kT}(z^*)),J_{c_kT}(z)-z^*\rangle=0.
\end{eqnarray}
Consequently, we have
\begin{eqnarray}
\label{tmpLemst}
&&\|z-z^*\|^2=\|J_{c_kT}(z)-z^*\|^2+\|z -J_{c_kT}(z)\|^2+2\langle z- J_{c_kT}(z), J_{c_kT}(z)-z^*\rangle\nonumber\\
&&=\|J_{c_kT}(z)-z^*\|^2+ \|c_k T(J_{c_kT}(z))\|^2+2\langle c_k T(J_{c_kT}(z)),J_{c_kT}(z)-z^*\rangle\nonumber\\
&&=\|J_{c_kT}(z)-z^*\|^2+ \|c_k T(J_{c_kT}(z))\|^2+2\langle c_k T(J_{c_kT}(z))-c_k T(J_{c_kT}(z^*)),J_{c_kT}(z)-z^*\rangle\nonumber\\
%&&=\|J_{c_kT}(z)-z^*\|^2+ \|c_k T(J_{c_kT}(z))-c_k T(z^*)\|^2+2\langle c_k T(J_{c_kT}(z))-c_k T(z^*),J_{c_kT}(z)-z^*\rangle\nonumber\\
&&=(1+\frac{c_k^2}{a^2})\|J_{c_kT}(z)-z^*\|^2,
\end{eqnarray}
in which the last inequality is because of the identity
$$\|c_k T(J_{c_kT}(z))-c_k T(J_{c_kT}(z^*))\|^2=\frac{c_k^2}{a^2}\|J_{c_kT}(z)-J_{c_kT}(z^*) \|^2$$
and the assertion (\ref{flesol}). Therefore, the inequality (\ref{new-1}) is tight and this indeed implies that the linear convergence rate to be established for (\ref{RPPAscheme-ck}) is optimal.

\end{remark}

\begin{lemma}
\label{neig}
Let $\{z^k\}$ be the sequence generated by the exact version of the generalized PPA scheme (\ref{RPPAscheme-ck}) with $\gamma\in (0,2)$, and $z^*$ be a solution point of (\ref{maxsol}). If $T^{-1}$ is Lipschitz continuous at $0$ with modulus $a>0$, and the proximal parameter sequence $\{c_k\}$ is bounded away from zero ($c_k\ge \kappa>0$ for any $k$), then there exists an integer $\hat k$ such that
\begin{equation}\label{new-3}
 \|\tilde z^k-z^*\|\le\frac{a}{\sqrt{a^2+c_k^2}}\|z^k-z^*\|\;\;\forall k>\hat k.
 \end{equation}
\end{lemma}
\begin{proof}
Applying Lemma \ref{strongm} with $z=z^k$, we know there exists $\tau>0$ such that
$$\|J_{c_kT}(z^k)-z^*\|\le \frac{a}{\sqrt{a^2+c_k^2}}\|z^k-z^*\|\;\; \mbox{when}\;\|c_k^{-1}(z^k-J_{c_kT}(z^k))\|\le \tau.$$
Using the notation $J_{c_kT}(z^k)=\tilde z^k$, this inequality can be rewritten as
$$\|\tilde z^k-z^*\|\le \frac{a}{\sqrt{a^2+c_k^2}}\|z^k-z^*\|\;\; \mbox{when}\;\|c_k^{-1}(z^k-\tilde z^k)\|\le \tau.$$
It follows from Theorem \ref{rel_PR global} that $\lim_{k\to \infty}\|z^k-\tilde z^k\|=0$. Then,
there exists an integer $\hat k$ such that,
$$c_k^{-1}\|z^k-\tilde z^k\|\le \kappa^{-1}\|z^k-\tilde z^k\|\le \tau\;\mbox{when} \;k>\hat k.$$
Thus, the assertion (\ref{new-3}) is implied by the two inequalities above. The proof is complete. $\hfill$
\end{proof}
\vspace{0.5cm}

Now, we prove the linear convergence rate of (\ref{RPPAscheme-ck}) in the following theorem.
\begin{theorem}
\label{rel_PR}
If $T^{-1}$ is Lipschitz continuous at $0$ with modulus $a>0$ and the proximal parameter $\{c_k\}$ is bounded away from zero ($c_k\ge \kappa>0)$, then the sequence $\{z^k\}$ generated by the exact version of the generalized PPA scheme (\ref{RPPAscheme-ck}) with $\gamma\in(0,2)$ satisfies
\begin{eqnarray}
\label{contl}
\|z^{k+1}-z^*\|^2\le\varrho\|z^k-z^*\|^2,
\end{eqnarray}
with
\begin{eqnarray}
\label{varrho}
\varrho:=1-\min(\gamma,2\gamma-\gamma^2)\frac{c_k^2}{a^2+c_k^2} \in (0,1).
%\left\{\begin{array}{ll}
% 1-\gamma\frac{c_k^2}{a^2+c_k^2},& 0<\gamma\le 1; \\
% 1-(2\gamma-\gamma^2)\frac{c_k^2}{a^2+c_k^2},&1<\gamma<2.
% \end{array}
%\right.
\end{eqnarray}
That is, the sequence $\{z^k\}$  converges linearly to a solution point of (\ref{maxsol}).
\end{theorem}
\begin{proof}
Simple algebra shows that
\begin{eqnarray*}
\|z^{k+1}-z^*\|^2&=&(1-\gamma)^2\|z^k-z^*\|^2+\gamma^2\|\tilde z^k-z^*\|^2+2\gamma(1-\gamma)\langle \tilde z^k-z^*, z^k-z^*\rangle\nonumber\\
                 &=&(1-\gamma)^2\|z^k-z^*\|^2+(2\gamma-\gamma^2)\|\tilde z^k-z^*\|^2+2\gamma(1-\gamma)\langle \tilde z^k-z^*,z^k-\tilde z^k\rangle.
\end{eqnarray*}
Obviously, the assertion (\ref{contl})-(\ref{varrho}) follows directly from Lemma \ref{neig} when $\gamma=1$. If $0<\gamma<1$, then it follows from Lemma \ref{neig} that
$$
\|z^{k+1}-z^*\|^2 \le (1-\gamma)\|z^k-z^*\|^2+\gamma\|\tilde z^k-z^*\|^2 =(1-\gamma\frac{c_k^2}{a^2+c_k^2})\|z^k-z^*\|^2.
$$
Moreover, if $1<\gamma<2$, because of (\ref{keytmp}) and Lemma \ref{neig}, we have
\begin{eqnarray*}
\|z^{k+1}-z^*\|^2\le\left(1-(2\gamma-\gamma^2)\frac{c_k^2}{a^2+c_k^2}\right)\|z^k-z^*\|^2.
\end{eqnarray*}
To show (\ref{varrho}), notice that $\gamma \in (0,2)$ and $c_k \ge \kappa>0$ for any $k$, and thus we have
  $$0<1-\min(\gamma,2\gamma-\gamma^2)\le \varrho:=1-\min(\gamma,2\gamma-\gamma^2)\frac{c_k^2}{a^2+c_k^2}
   <1-\min(\gamma,2\gamma-\gamma^2)\frac{\kappa^2}{a^2+\kappa^2}<1.$$
Thus, the inequalities (\ref{contl}) and (\ref{varrho}) imply the linear convergence rate of the sequence $\{z^k\}$. The proof is complete. $\hfill$
\end{proof}
\vspace{0.5cm}

\begin{remark}
The proof of Theorem \ref{rel_PR} shows that because of the tightness of the inequality (\ref{new-1}), the inequality (\ref{contl}) cannot be improved in the sense that no constant smaller than $\varrho$ defined in (\ref{varrho}) can be found such that the inequality (\ref{contl}) still holds. Thus, the linear convergence of the PPA scheme (\ref{RPPAscheme-ck}) established in Theorem \ref{rel_PR} is optimal.
\end{remark}

\begin{remark}\label{relassuA}
Similarly as Definition \ref{invLipze}, if a sequence $\{z^k\}$ converges to $z^*$, we can define ``$T^{-1}$ is  Lipschitz continuous with modulus $a\ge0$ at the sequence $\{z^k\}$" if there exists some $\tau>0$ such that
 $$\|z^k-z^*\|\le a\|w^k\|\;\mbox{whenever}\; z^k\in T^{-1}(w^k)\;\mbox{and}\;\|w^k\|\le \tau.$$
Then, it can be easily seen that the linear convergence of $\{z^k\}$ generated by (\ref{RPPAscheme-ck}) can be guaranteed under the less strengthen assumption `` $T^{-1}$ is Lipschitz continuous at the iterates $\{\tilde z^k\}$ with positive modulus when $k$ is sufficiently large". Recall the fact  $\tilde z^k\in T^{-1}(c_k^{-1}(z^k-\tilde z^k))$ and $z^*\in T^{-1}(0)$. Then, this less strengthen assumption is equivalent to saying that there exists an integer $\hat k$ such that
\begin{eqnarray}
\label{relassumptA}
\|\tilde z^k-z^*\|\le a \|c_k^{-1}(z^k-\tilde z^k)\|\;\mbox{when}\;k>\hat k,
\end{eqnarray}
where $\hat k$ is large enough such that $\|c_k^{-1}(z^k-\tilde z^k)\| \le \tau$. Note that $\|c_k^{-1}(z^k-\tilde z^k)\| \le \tau$ can be ensured by the fact $\lim_{k \to \infty}\|z^k-{\tilde z}^k\|=0$ implied in (\ref{cont_PR}) and that $\{c_k\}$ is bounded away from $0$. More discussion is referred to Section \ref{relCon}.
\end{remark}

\section{The Convergence of the Inexact Version (\ref{RPPAscheme-ck-inexact})}
\label{inexPPA}

In this section, we specify the inexactness criterion for (\ref{RPPAscheme-ck-inexact}) and show its linear convergence under the same assumption of ``$T^{-1}$ is Lipschitz continuous at $0$ with positive modulus". This is a generalization of the inexact version (\ref{PPAscheme-inexact}) considered in \cite{Rock76}. More specifically, we consider the scheme
\begin{eqnarray}
\label{InexaGPPA}
\left\{\begin{array}{l}
z^{k+1}=(1-\gamma)z^k+\gamma \bar z^k,\\
\|\bar z^k-J_{c_kT}(z^k)\|\le \delta_k\|z^k-z^{k+1}\|,
\end{array}\right.
\end{eqnarray}
where $\gamma\in(0,2)$, $c_k \ge \kappa>0$ for any $k$, and $\{\delta_k\}$ is a sequence of positive real numbers satisfying $\sum_k \delta_k<+\infty$.

Note that in (\ref{InexaGPPA}), we consider using relative errors to control the accuracy in (\ref{RPPAscheme-ck-inexact}); thus it is different from the inexact version in \cite{CYuan} which uses absolute errors. We still use the notation $\tilde z^k=J_{c_kT}(z^k)$ in the upcoming analysis.

\subsection{Global Convergence}
\label{inexactglocon}
Again, we first show the global convergence for the sequence $\{z^k\}$ generated by (\ref{InexaGPPA}). Note that we do not need the assumption ``$T^{-1}$ is Lipschitz continuous at $0$ with positive modulus" for proving the global convergence. We first prove several lemmas for this purpose. Their proofs are elementary; but we still include them for completeness.

\begin{lemma}
\label{sumprod}
Let $\{\alpha_k\}$ be a positive sequence satisfying $\sum_{k=1}^{\infty} \alpha_k<+\infty$. Then, we have
$$
\prod_{k=1}^{\infty} (1+\alpha_k)<+\infty.
$$
\end{lemma}
\begin{proof}
Obviously, it holds that $\log(1+x)\le x\; \; \mbox{when}\;\; 0<x<1$. Hence, we have
$$\sum_{k=1}^{\infty}\log (1+\alpha_k)\le \sum_{k=1}^{\infty} \alpha_k<+\infty,$$
which implies the assertion immediately. $\hfill$ \end{proof}

\begin{lemma}
\label{furprod}
Let $\{\delta_k\}$ be a positive sequence satisfying
$\sum_{k=1}^{\infty} \delta_k<+\infty$ and $\gamma>0$ be a constant. Then, we have
$$\prod_{k=1}^{\infty} \frac{1+\gamma\delta_k}{1-\gamma\delta_k}<+\infty.$$
\end{lemma}

\begin{proof}
Since $\sum_{k=1}^{\infty} \delta_k<+\infty$, we have $\delta_k\rightarrow 0$ when $k \to \infty$.
Thus, there exists an integer $\hat k$ such that
\begin{eqnarray}
\label{keylemma2}
1-\gamma\delta_k\ge \frac{1}{2}\;\;\mbox{when}\;\; k\ge\hat k.
\end{eqnarray}
Hence, we have
\begin{eqnarray*}
\prod_{k=1}^{\infty} \frac{1+\gamma\delta_k}{1-\gamma\delta_k}=\prod_{k=1}^{\hat k-1} \frac{1+\gamma\delta_k}{1-\gamma\delta_k}
\cdot \prod_{k=\hat k}^{\infty} \frac{1+\gamma\delta_k}{1-\gamma\delta_k}\le \prod_{k=1}^{\hat k-1} \frac{1+\gamma\delta_k}{1-\gamma\delta_k}
\cdot 2\prod_{k=\hat k}^{\infty}(1+\gamma\delta_k)<+\infty.
\end{eqnarray*}
The proof is complete. $\hfill$
\end{proof}

\begin{lemma}
\label{existnonmono}
Let $\{a_k\}$ and $\{b_k\}$ be positive sequences; $\sum_{k=1}^{\infty} b_k<+\infty$; and
\begin{eqnarray}
\label{mono}
a_{k+1}\le a_k+b_k,\;\;\forall k.
\end{eqnarray}
Then, the sequence $\{a_k\}$ is convergent.
\end{lemma}
\begin{proof}
First, it follows from (\ref{mono}) that
$$a_{k+1}\le a_1+\sum_{i=1}^k b_i \le a_1+\sum_{i=1}^{\infty} b_i, \;\;\forall \;k.$$
Since $\sum_{k=1}^{\infty} b_k<+\infty$, the sequence $\{a_k\}$ is bounded. Thus, it has at least
one accumulation point, say $a_1^*$. That is, there exists a subsequence $\{a_{k_j}\}$ converging to $a_1^*$.
Suppose that the sequence $\{a_k\}$ is not convergent.
Then, there exists another subsequence $\{a_{k_t}\}$ converging to another point, say $a_2^*$. Obviously, $a_1^*\neq a_2^*$.
Without loss of generality, let us assume $a_2^*>a_1^*$. Define $\epsilon=\frac{1}{2}(a_2^*-a_1^*)>0$.
There exists an integer $J_2$ such that
$$\sum_{i=J_2}^{\infty} b_i<\epsilon,$$
where $\epsilon>0$ is a given scalar. On the other hand, for the given $\epsilon$, there exists integers $J_1>J_2$ such that
$$|a_{k_{J_1}}-a_1^*|<\epsilon.$$
  Then, we get
$$a_{k_t}\le a_{k_{J_1}}+\sum_{i=k_{J_1}+1}^{\infty}b_i<a_1^*+\epsilon+\epsilon=a_2^*, \;\; \forall k_t>k_{J_1}.$$
It contradicts with the fact $a_{k_t}\rightarrow a_2^*$ when $t \to \infty$. Hence, the sequence $\{a_k\}$ is convergent. The proof is complete. $\hfill$
\end{proof}
%[{\bf I have revised the proof, check it. Done}]

Now we start to prove the global convergence of (\ref{InexaGPPA}). The key is the sequence generated by the inexact version (\ref{InexaGPPA}) is asymptotically emerged with the sequence by the generalized PPA (\ref{RPPAscheme-ck}). With this fact, the convergence of (\ref{InexaGPPA}) can be established easily.

\begin{theorem}\label{new-2}
Let $\{z^k\}$ be the sequence generated by the inexact version of the generalized PPA scheme (\ref{InexaGPPA}). Then, we have
\begin{itemize}
\item [(1).] The sequence $\{z^k\}$ is bounded.

\item[(2).] It holds that
\begin{eqnarray}
\label{solcon}
\lim_{k\to \infty}\|z^k-\tilde z^k\|=0.
\end{eqnarray}
\end{itemize}
\end{theorem}
\begin{proof}
Recall we use ${\tilde z}^k=J_{c_kT}(z^k)$ for easier notation. Let us use one more notation
$$
\hat z^{k+1}:=(1-\gamma)z^k+\gamma \tilde z^k.
$$
Indeed, $\hat z^{k+1}$ denotes the iterate generated by the exact version (\ref{RPPAscheme-ck}) from the given $z^k$. Thus, for an arbitrary solution point $z^*$ of (\ref{maxsol}), it follows from (\ref{cont_PR}) that
\begin{eqnarray}
\label{inexatmp1}
\|\hat z^{k+1} -z^*\|^2\le\|z^k-z^*\|^2-\gamma(2-\gamma)\|z^k-\tilde z^k\|.
\end{eqnarray}
Recall the definition of $z^{k+1}$ in (\ref{InexaGPPA}). We have
\begin{eqnarray}
\label{tkp1}
\hat z^{k+1}-z^{k+1}=\gamma(\tilde z^k-\bar z^k),
\end{eqnarray}
\;where $\bar z^k$ is also given in (\ref{InexaGPPA}). Thus, for any solution point $z^*$ of (\ref{maxsol}), we have
\begin{eqnarray}\label{inexatmp2}
\|z^{k+1}-z^*\|&\le &\|z^{k+1}-\hat z^{k+1}\|+\|\hat z^{k+1}-z^*\|\nn\\
                 &\le& \gamma\delta_k\|z^k-z^{k+1}\|+\|\hat z^{k+1}-z^*\|\nn\\
                 &\le& \gamma\delta_k(\|z^k-z^*\|+\|z^{k+1}-z^*\|)+\|\hat z^{k+1}-z^*\|\nn\\
                 &\le& \gamma\delta_k(\|z^k-z^*\|+\|z^{k+1}-z^*\|)+\|z^k-z^*\|,
\end{eqnarray}
where the second inequality results from the inexact criterion in (\ref{InexaGPPA}) and the last
inequality follows from (\ref{inexatmp1}).
Then, we get
$$\|z^{k+1}-z^*\|\le \frac{1+\gamma\delta_k}{1-\gamma\delta_k}\|z^k-z^*\|\le \cdots \le \prod_{i=1}^{k} \frac{1+\gamma\delta_i}{1-\gamma\delta_i}\|z^0-z^*\|.$$
Using Lemma \ref{furprod}, the sequence $\{z^k\}$ is bounded. The first assertion is proved.

Now we prove the second assertion. Again, for an arbitrary solution point $z^*$ of (\ref{maxsol}), since $\{z^k\}$ is bounded and because of (\ref{inexatmp1}), there exists a positive scalar $R$ such that
\begin{eqnarray}
\label{tkp3}
\|z^k-z^*\| < R, \;\;\forall k
\end{eqnarray}
and
\begin{eqnarray}
\label{tkp2}
\|\hat z^k-z^*\|<R, \;\;\forall k.
\end{eqnarray}
We thus have
\begin{eqnarray}
\label{inexatmp3}
\|z^{k+1}-z^*\|^2&=&\|\hat z^{k+1}-z^*+(z^{k+1}-\hat z^{k+1})\|^2\nn\\
                 &=&\|\hat z^{k+1}-z^*\|^2+\|z^{k+1}-\hat z^{k+1}\|^2+2\langle \hat z^{k+1}-z^*, z^{k+1}-\hat z^{k+1}\rangle \nn \\
                 &\le&  \|\hat z^{k+1}-z^*\|^2 + 2\|\hat z^{k+1}-z^* \|\| z^{k+1}-\hat z^{k+1}\| + \|z^{k+1}-\hat z^{k+1}\|^2\nn\\
                 &\le&  \|\hat z^{k+1}-z^*\|^2+2R\gamma\delta_k\|z^k-z^{k+1}\|+\gamma^2\delta^2_k \|z^k-z^{k+1}\|^2\nn\\
                 &\le&   \|z^k-z^*\|^2-\gamma(2-\gamma)\|z^k-\tilde z^k\|^2 +4R^2\gamma\delta_k+4R^2\gamma^2\delta_k^2,
\end{eqnarray}
where the second inequality follows from (\ref{tkp2}) and (\ref{InexaGPPA}); and the last inequality is because of (\ref{inexatmp1}) and
(\ref{tkp3}).
Moreover, since $\gamma \in (0,2)$, we have
$$\|z^{k+1}-z^*\|^2\le\|z^k-z^*\|^2+4R^2\gamma\delta_k+4R^2\gamma^2\delta_k^2,$$
and
$$\sum_k (4R^2\gamma\delta_k+4R^2\gamma^2\delta_k^2)<+\infty.$$
Now, using Lemma \ref{existnonmono} with $a_k:=\|z^k-z^*\|^2$ and $b_k:=4R^2\gamma\delta_k+4R^2\gamma^2\delta_k^2$, we obtain
\begin{eqnarray}
\label{inbouzk}
\lim_{k\rightarrow \infty}\|z^k-z^*\|=:A,\;\;
\end{eqnarray}
where $A$ is a positive scalar. On the other hand, recall that $\{\delta_k\}$ is summable, so is $\{\delta_k^2\}$. We thus have $\sum_k\delta_k^2<\infty$.
Then, it follows from (\ref{inexatmp3}) that
$$\gamma(2-\gamma)\|z^k-\tilde z^k\|^2\le \|z^k-z^*\|^2-  \|z^{k+1}-z^*\|^2+4R^2\gamma\delta_k+4R^2\gamma^2\delta_k^2,\; \forall k.$$
Then, we have $\sum_k\|z^k-\tilde z^k\|^2<\infty$ and thus $\lim_{k\to \infty}\|z^k-\tilde z^k\|=0$.
The proof is complete. $\hfill$
\end{proof}

Theorem \ref{new-2} shows that the accuracy of iterates generated by the inexact version (\ref{InexaGPPA}) is iteratively increased, which essentially implies the convergence of the sequence of (\ref{InexaGPPA}). We provide the rigorous proof in the following theorem.

\begin{theorem}
\label{gloinGPPA}
(Global convergence) Let $\{z^k\}$ be the sequence generated by the inexact version of the generalized PPA scheme (\ref{InexaGPPA}). Then, it converges to a solution point of (\ref{maxsol}).
 %If $T^{-1}$ is Lipschitz continuous at $0$, then $\{z^k\}$ converges strongly to $z^*$, the unique solution of $0\in T(z)$.
\end{theorem}
\begin{proof}
Since the sequence $\{z^k\}$ is bounded, it has an accumulation point $z^{\infty}$. Let $\{z_{k_j}\}$ be the subsequence converging to $z^{\infty}$. That is,
$\lim_{j \to \infty}\|z^{k_j}-z^\infty\|= 0$. Using (\ref{solcon}) with $k=k_j$, we have
\begin{eqnarray}\label{inter0}
\|\tilde z^{k_j} -z^{k_j}\|\rightarrow 0,\;\;\hbox{when}\;\; j \to \infty.
\end{eqnarray}
Then, combining with $\|z^{k_j}-z^\infty\|\rightarrow 0$,  we get
 \begin{eqnarray}
 \label{inter1}
 \|J_{c_{k_j}T}(z^{k_j})-z^\infty\|\rightarrow0, \;\;\hbox{when}\;\; j \to \infty.
 \end{eqnarray}
Also, it follows from (\ref{inter0}) that
\begin{eqnarray}
\label{inter2}
\|z^{k_j}-J_{c_{k_j} T}(z^{k_j})\|\rightarrow0, \;\;\hbox{when}\;\; j \to \infty.
\end{eqnarray}
Note that
$$c_k^{-1}(z^k-J_{c_kT}(z^k))\in T(J_{c_k T}(z^k)).$$
Thus, using the monotonicity of $T$, for any $k$, we have
$$
\langle z-J_{c_k T}(z^k), w- c_k^{-1}(z^k-J_{c_kT}(z^k))\rangle\ge0, \;\; \hbox{for all} \; z, w \;\hbox{satisfying}\;\; w \in T(z). $$
Let $k=k_j$ in the above inequality, take $j\rightarrow\infty$, and combine it with (\ref{inter1}) and (\ref{inter2}). We thus have
$$
\langle z-z^\infty, w \rangle \ge0, \;\; \hbox{for all} \; z, w \;\hbox{satisfying}\;\; w \in T(z).$$
which, together with the monotonicity of $T$, means that $z^\infty$ is a solution point of (\ref{maxsol}).

Finally, since $z^\infty$ is a solution point of (\ref{maxsol}), (\ref{inbouzk}) can be written as $\lim_{k\rightarrow \infty}\|z^k-z^{\infty}\|=A$ and indeed we have $A=0$ because $z^{k_j} \to z^\infty$. Thus, the sequence $\{z^k\}$ converges to $z^\infty$ which is a solution point of (\ref{maxsol}). The proof is complete. $\hfill$
\end{proof}

\subsection{Linear Convergence}
\label{inexactlincon}
In this subsection, under the assumption ``$T^{-1}$ is Lipschitz continuous at $0$ with positive modulus", we prove the linear convergence for the sequence $\{z^k\}$ generated by (\ref{InexaGPPA}). Recall the notation $\tilde z^k=J_{c_kT}(z^k)$. We first prove a lemma.

\begin{lemma}
\label{neiginexa} Let $\{z^k\}$ be the sequence generated by the inexact version of the generalized PPA scheme (\ref{InexaGPPA}) with $\gamma\in(0,2)$ and $\sum_k\delta_k<+\infty$.
If $T^{-1}$ is Lipschitz continuous at $0$ with modulus $a>0$,
then there exists an integer $k_1$ such that
$$ \|\tilde z^k-z^*\|\le\frac{a}{\sqrt{a^2+c_k^2}}\|z^k-z^*\|\;\;\forall k>k_1.$$
\end{lemma}
\begin{proof}
%Note that the definition of $\tilde z^k$ is the same as Lemma \ref{neig}, the conclusion follows directly.
From Theorem \ref{new-2}, we know that $\lim_{k\rightarrow 0}\|z^k-\tilde z^k\|=0$.
Then, the conclusion follows immediately from the proof of Lemma \ref{neig}. $\hfill$
\end{proof}

The main result of this subsection is summarized in the following theorem. This result reduces to Theorem 2 in \cite{Rock76} if $\gamma=1$.

\begin{theorem}
\label{rel PRinexact}
Assume $T^{-1}$ is Lipschitz continuous at $0$ with modulus $a>0$ and the proximal parameter sequence $\{c_k\}$ is bounded away from zero ($c_k\ge \kappa>0)$. Let $\{z^k\}$ be the sequence generated by the inexact version of the generalized PPA scheme (\ref{InexaGPPA}). Then, there exist an integer $\hat{k}$ such that
$$\|z^{k+1}-z^*\|\le \theta_k\|z^k-z^*\|\;\;\hbox{when}\;\;k> \hat{k},
$$
where $z^*$ is a solution point of (\ref{maxsol}) and
$$
0<\theta_k:=\frac{\sqrt{\left(1-\min(\gamma,2\gamma-\gamma^2)\frac{c_k^2}{a^2+c_k^2}\right)}+\gamma\delta_k}{1-\gamma\delta_k}<1, \;\hbox{when}\; k > \hat{k}.$$
That is, $\{z^k\}$ converges linearly to $z^*$.
\end{theorem}
\begin{proof}
Recall in Theorem \ref{gloinGPPA}, it is proved that the sequence $\{z^k\}$ converges to a solution point $z^*$ of (\ref{maxsol}). First, it is easy to see that there exists an integer $ k_1$ such that
%\begin{eqnarray*}
%\|z^k-\tilde z^k\|&=&\|z^k-z^*+z^*-{\tilde z}^k\| \\
%&\ge& \|z^k-z^*\|-\|z^*-{\tilde z}^k\| \\
%&\ge&(1-\frac{a}{\sqrt{a^2+c_k^2}})\|z^k-z^*\|,\;\;k> k_1,
%\end{eqnarray*}
%where the last inequality follows from Lemma \ref{neiginexa}. Then, using (\ref{inexatmp1}), we have
\begin{eqnarray}
\label{kkeyi}
\|\hat z^{k+1} -z^*\|^2\le \left(1-\min(\gamma,2\gamma-\gamma^2)\frac{c_k^2}{a^2+c_k^2}\right)\|z^k-z^*\|^2,\;\;k> k_1.
\end{eqnarray}
In addition, it follows from (\ref{inexatmp2}) that
\begin{eqnarray*}
\|z^{k+1}-z^*\|&\le& \gamma\delta_k(\|z^k-z^*\|+\|z^{k+1}-z^*\|)+\|\hat z^{k+1}-z^*\|\nn\\
               &\le &\gamma\delta_k(\|z^k-z^*\|+\|z^{k+1}-z^*\|)+\sqrt{\left(1-\min(\gamma,2\gamma-\gamma^2)\frac{c_k^2}{a^2+c_k^2}\right)}\|z^k-z^*\|,\;\;k> k_1,
\end{eqnarray*}
where the last inequality follows from  (\ref{kkeyi}). Accordingly,  we have
$$\|z^{k+1}-z^*\|\le \frac{\sqrt{1-\min(\gamma,2\gamma-\gamma^2)\frac{c_k^2}{a^2+c_k^2}}+\gamma\delta_k}{1-\gamma\delta_k}\|z^k-z^*\| \;\;\mbox{when}\;k> k_1.$$
Note that $\delta_k\rightarrow 0$ and $c_k\ge\kappa>0$. Then, there exists an integer $\hat{k}$, without loss of generality, assuming $\hat{k}>k_1$, such that
$$\theta_k:=\frac{\sqrt{1-\min(\gamma,2\gamma-\gamma^2)\frac{c_k^2}{a^2+c_k^2}}+\gamma\delta_k}{1-\gamma\delta_k}<1, \;\;\mbox{when}\;\;k> \hat{k}.$$
Hence, $\{z^k\}$ converges linearly to $z^*$, a solution point of (\ref{maxsol}). The proof is complete. $\hfill$
\end{proof}

\begin{remark}\label{relassuA-2}
Similarly as Section \ref{linear-exact}, it can been seen from the proofs of Lemma
\ref{neiginexa} and Theorem \ref{rel PRinexact} that the linear convergence of the sequence $\{z^k\}$ generated by (\ref{InexaGPPA}) can be guaranteed under the less strengthen condition  `` $T^{-1}$ is Lipschitz continuous at the iterates $\{\tilde z^k\}$ with positive modulus when $k$ is large enough".
\end{remark}

\section{Further Study on Assumption} \label{relCon}
\setcounter{equation}{0}

Under the assumption ``$T^{-1}$ is Lipschitz continuous at $0$ with positive modulus", we have shown the linear convergence  for both the exact version (\ref{RPPAscheme-ck}) and inexact version (\ref{InexaGPPA}) of the generalized PPA. Recall that the generalized PPA (\ref{RPPAscheme-ck}) include the PPA  (\ref{PPAscheme}) as a special case with $\gamma=1$ and our analysis extends the result in \cite{Rock76} for (\ref{PPAscheme}). In \cite{CYuan}, the linear convergence of the generalized PPA (\ref{RPPAscheme-ck}) with $c_k\equiv c$ has been studied under the assumption that $T$ is $\alpha$-strongly monotone, which is stronger than ``$T^{-1}$ is Lipschitz continuous at $0$ with positive modulus".

In the following, we show that although we restrict our analysis under the assumption ``$T^{-1}$ is Lipschitz continuous at $0$ with positive modulus", theoretically this assumption can be further relaxed in order to ensure the linear convergence of (\ref{RPPAscheme-ck}) and (\ref{InexaGPPA}). Note that the assertion in the following lemma does not depend on any specific iterative scheme.

\begin{theorem}\label{Thm-new-1}
Let $z^*$ be a solution point of (\ref{maxsol}) and the sequence $\{c_k\}$ be both upper and below bounded, i.e., $0<\kappa\le c_k \le \zeta$ for all $k$.
If $T^{-1}$ is Lipschitz continuous at 0 with positive modulus $a$, then $J_{c_k T}$ defined in (\ref{resolvent}) is Lipschitz continuous at $z^*$ and $\sup_k \{L_k\}<1$, where $L_k$ is the Lipschitz constant of $J_{c_k T}$.
\end{theorem}
\begin{proof}
%First, according to Theorem \ref{rel_PR global} and \ref{new-2}, we have
%  $\lim_{k\rightarrow \infty}\|z^k-\tilde z^k\|=0$.
%  Then, there exists an integer $\hat k$ such that,
%  $$c_k^{-1}\|z^k-\tilde z^k\|\le \kappa^{-1}\|z^k-\tilde z^k\|\le \tau\;\mbox{when} \;k>\hat k.$$
It follows from  Lemma \ref{strongm} that there exists $\tau>0$ such that
\begin{eqnarray}
\label{newtmpn}
\|J_{c_kT}(z)-z^*\|\le \frac{a}{\sqrt{a^2+c_k^2}}\|z-z^*\|,\;\; \mbox{when}\;\|c_k^{-1}(z-J_{c_kT}(z))\|\le \tau.
\end{eqnarray}
Recall that $z^*\in T^{-1}(0)$, $J_{c_kT}(z)\in T^{-1}(c_k^{-1}(z-J_{c_kT}(z))$ and  $T^{-1}$ is Lipschitz continuous at 0. We thus have
$$ \|J_{c_kT}(z)- z^*\|\rightarrow 0, \;\;\mbox{when}\;\; \|c_k^{-1}(z-J_{c_kT}(z))\|\rightarrow 0.$$
Since $c_k \le \zeta$ for all $k$, we have $\frac{1}{\zeta}\|J_{\zeta T}(z)- z\|\le\frac{1}{c_k}\|J_{c_kT}(z)- z\|\rightarrow 0$.
From above inequality, we see that $\|z-z^*\| \rightarrow0$ when
$c_k^{-1}\|z-J_{c_kT}(z)\|\rightarrow 0$.
Thus, $J_{c_k T}$ is Lipschitz continuous at $z^*$ with the constant $L_{k}:=\frac{a}{\sqrt{a^2+c_k^2}}\le \frac{a}{\sqrt{a^2+\kappa^2}}<1$ for any $k$, according to (\ref{newtmpn}). The proof is complete.
%[{\bf Your previous writing of $\|T(u)\| \cdots \cdots$ is wrong. Also, I see no reason why use the notation $G_k$ so I remove it. Check this proof and make it rigorous}Thanks!
%My idea is as follows: when $z$ is near $z^*$, then $\frac{1}{\kappa}\|z-J_{\kappa T}(z)\|\rightarrow 0$.
%Hence, the condition ($c_k^{-1}\|z-J_{c_kT}(z)\|\le \tau$) is satisfied. Then, we get $J_{c_k T}$ defined in (\ref{resolvent}) is Lipschitz continuous at $z^*$ and $\sup_k \{L_k\}<1$ according to the first inequality in the proof.].
\end{proof}

\begin{theorem} \label{rela}
Suppose the sequence $\{c_k\}$ is both upper and below bounded, {\it that is}, there exists constants $\varsigma$ and $\kappa$ such that $0<\kappa\le c_k\le\varsigma$ for all $k$. Let  $\{z^k\}$ be the sequence generated by the exact version of the generalized PPA (\ref{RPPAscheme-ck}) or the inexact version (\ref{InexaGPPA}).
 If $J_{c_k T}$ is Lipschitz continuous at $z^*$ with the constant $L_k$, and $ L_G:=\sup_k\{L_k\}<1$, then
  \begin{itemize}
  \item[(1)] $T^{-1}$ is Lipschitz continuous at all the iterates $\{\tilde z^k\}$  with positive modulus when $k$ is sufficiently large.
  \item[(2)] $\{z^k\}$ converges linearly to a solution point of (\ref{maxsol}).
  \end{itemize}
\end{theorem}

\begin{proof}
%Recall that $\{z^k\}$ be generated by the exact version of the generalized PPA (\ref{RPPAscheme-ck}) or the inexact version (\ref{InexaGPPA}).
For a solution point of (\ref{maxsol}), $z^*$, we have $z^*=J_{c_k T}(z^*)$. Recall the notation $\tilde z^k=J_{c_kT}(z^k)$. Thus, it holds that
\begin{equation}\label{LGPR}
\|z^k-\tilde z^k\|=\|(z^k-z^*)-(J_{c_k T}(z^k)-J_{c_k T}(z^*))\| \ge\|z^k-z^*\|-\|J_{c_k T}(z^k)-J_{c_k T}(z^*)\| \ge(1-L_G)\|z^k-z^* \|,
\end{equation}
which implies
\begin{eqnarray}
\label{GLinter1}
\frac{1}{(1-L_G)^2}\|z^k-\tilde z^k\|^2\ge\|z^k-z^*\|^2\ge\|\tilde z^{k}-z^*\|^2+\|z^k- \tilde z^{k}\|^2.
\end{eqnarray}
Then, it follows from the above inequality and $0<c_k\le\varsigma$ that
\begin{eqnarray}\label{tmpnew}
\|\tilde z^k-z^*\|^2\le \frac{2L_G-L_G^2}{(1-L_G)^2}\|z^k-\tilde z^k\|^2\le\frac{2L_G-L_G^2}{(1-L_G)^2}\varsigma^2\|c_k^{-1}(z^k-\tilde z^k)\|^2.
\end{eqnarray}
%[{\bf it is not correct to write $\|T(\tilde z^k)-T(z^*)\|^2$ -- the distance of two sets!}]
According to Theorems \ref{rel_PR global} and \ref{new-2}, for the sequence $\{z^k\}$ generated by either the exact version (\ref{RPPAscheme-ck}) or the inexact version (\ref{InexaGPPA}), we have
  $\lim_{k\rightarrow \infty}\|z^k-\tilde z^k\|=0$.
Since $c_k \ge \kappa>0$, there exists an integer $\hat k$ such that
  $$\|c_k^{-1}(z^k-\tilde z^k)\|\le \kappa^{-1}\|z^k-\tilde z^k\|\le \tau\;\mbox{when} \;k>\hat k,$$
where $\tau>0$ is a given constant.
%Hence $\|T(\tilde z^k)\|\rightarrow 0$ when $k$ is large enough
%[{\bf Same problem!}].
Note the facts $\tilde z^k\in T^{-1}(c_k^{-1}(z^k-\tilde z^k))$ and $z^*\in T^{-1}(0)$.
Consequently, it follows from (\ref{tmpnew}) that $T^{-1}$ is Lipschitz continuous at all the iterates $\{\tilde z^k\}$ with modulus $a:=\frac{\varsigma\sqrt{2L_G-L_G^2}}{1-L_G}$ when $k$ is large enough.

Now, we prove (2). Indeed, as commented in Remarks \ref{relassuA} and \ref{relassuA-2}, the linear convergence of the schemes (\ref{RPPAscheme-ck}) and (\ref{InexaGPPA}) can be ensured since $T^{-1}$ is Lipschitz continuous at all the iterates $\{\tilde z^k\}$  with positive modulus
 when $k$ is sufficiently large and $\{c_k\}$ is below bounded. Thus, the assertion (2) is proved. The proof is complete.  $\hfill$
%From Remark \ref{relassuA}, we know that the relaxed {\bf Assumption A} holds.
\end{proof}
%\begin{remark}\label{remk3}
%Theorems \ref{rela} shows that the linear convergence of the schemes (\ref{RPPAscheme-ck}) and (\ref{InexaGPPA}) can be ensured under conditions more relaxed than ``$T^{-1}$ is Lipschitz continuous at $0$". In particular, if we take $c_k\equiv c$, the sufficient condition of ensure the linear convergence reduces to $L_G<1$ with $G_k\equiv G$.
%\end{remark}

%Theorems \ref{Thm-new-1} and \ref{rela} show that the assumption ``$T^{-1}$ is Lipschitz continuous at $0$" is sufficient, and can be further relaxed, to ensure the linear convergence rate for the  exact version (\ref{RPPAscheme-ck}) and inexact version of the generalized PPA (\ref{InexaGPPA}). An interesting corollary relevant to Theorem \ref{Thm-new-1} is stated below.

%{\color{blue}Therefore, theoretically it is still possible to relax the condition ``$T^{-1}$ is Lipschitz continuous at $0$ with positive modulus" in \cite{Rock76} for studying the linear convergence rate of PPA-based schemes, though such a analysis is not covered in this paper.}

So far, we have mentioned various conditions including  strongly convexity in \cite{CYuan}, the assumption in \cite{Rock76} and the one in Theorems \ref{Thm-new-1} and \ref{rela},  to ensure the linear convergence of the schemes (\ref{RPPAscheme-ck}) and (\ref{InexaGPPA}). In Figure  \ref{RelationCond}, we show their relationships for the special case where $c_k\equiv c$ for all $k$, which is clearly an interesting choice for implementing the schemes (\ref{RPPAscheme-ck}) and (\ref{InexaGPPA}).

\begin{figure}[t]
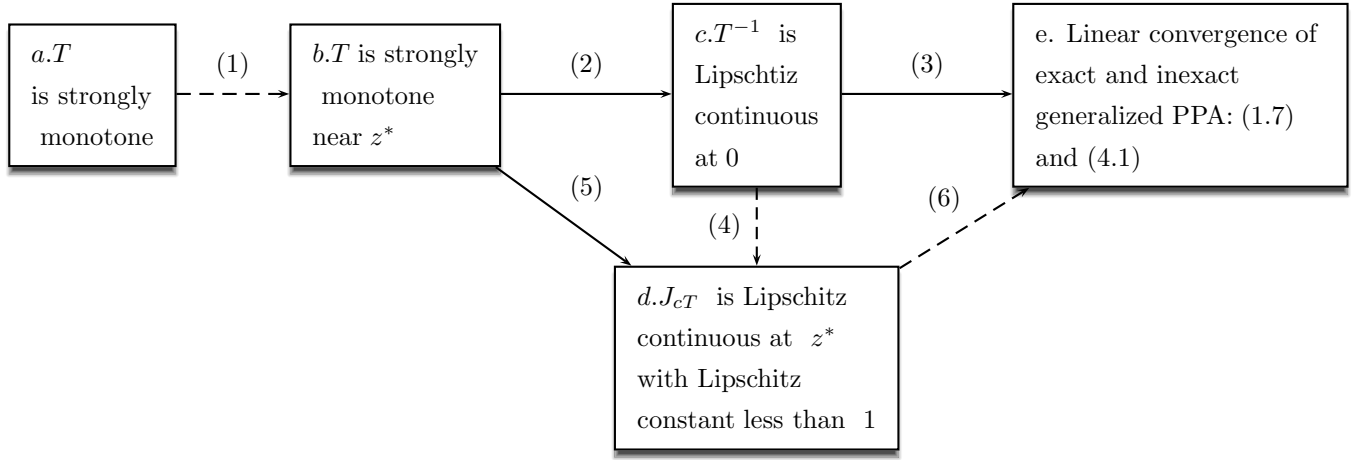

\centering
\psset{shadowcolor=black!70,shadowangle=-90,blur,shortput=nab}
\begin{psmatrix}[rowsep=1,colsep=1.5]
\psframebox[shadow=true]{$\begin{array}{l}
a. T\\ \hbox{is strongly}\\
\hbox{ monotone}
%\footnote{The defintion of coercive is the strongly convex in \cite{CYuan}}
\end{array}$}
&\psframebox[shadow=true]{$\begin{array}{l}
b. T \;\hbox{is strongly}\\
\hbox{ monotone } \\
 \hbox{near}\;  z^*\\
\end{array}$}
&\psframebox[shadow=true]{$\begin{array}{l}c. T^{-1}\ \ \hbox{is}\\
\hbox{Lipschtiz} \\
\hbox{continuous} \\ \hbox{at}\;0\end{array}$}
&\psframebox[shadow=true]{$\begin{array}{l}\hbox{e. Linear convergence of }\\
\hbox{exact and inexact }\\
\hbox{generalized PPA:}\; (\ref{RPPAscheme-ck})\\
 \hbox{and (\ref{InexaGPPA})}\end{array}$}
\\
&&\psframebox[shadow=true]{$\begin{array}{l}d. J_{cT}\ \ \hbox{is Lipschitz}\\
%\hbox{} \\
\hbox{continuous at} \ \ z^*\\
\hbox{with Lipschitz}\\
\hbox{constant less than}\ \ 1\end{array}$}
\end{psmatrix}
\ncline[arrows=->,linestyle=dashed]{1,1}{1,2}
\naput{(1)}
\ncline{->}{1,2}{1,3}
\naput{(2)}
\ncline{->}{1,3}{1,4}
\naput{(3)}
\ncline{->}{1,2}{2,3}
\naput{(5)}
\ncline[arrows=<-,linestyle=dashed]{2,3}{1,3}
\naput{(4)}
\ncline[arrows=->,linestyle=dashed]{2,3}{1,4}
\naput{(6)}
\caption{Relationships among different conditions for linear convergence of (\ref{RPPAscheme-ck}) and (\ref{InexaGPPA}).}\label{RelationCond}
\end{figure}

\section{Discuss on the Superlinear Convergence}\label{super}
\setcounter{equation}{0}

In \cite{Rock76}, under the assumption that`` $T^{-1}$ is Lipschitz continuous at $0$  with positive modulus", it was shown that the special case of (\ref{InexaGPPA}) with $\gamma=1$ is superlinearly convergent if the proximal parameter $c_k \to \infty$. See Theorem 2 in \cite{Rock76}. One may ask if we can extend the same superlinear convergence result to (\ref{InexaGPPA}) with a general $\gamma$ in $(0,2)$. In this section, we take a closer look at this issue and give a negative answer to this question. It is sufficient to just analyze the exact version (\ref{RPPAscheme-ck}) to answer this question.

Recall (\ref{contl}) and (\ref{varrho}). We have
\begin{eqnarray}
\label{tmp2}
\|z^{k+1}-z^*\|^2&\le& \left(1-\min(\gamma,2\gamma-\gamma^2)\frac{c_k^2}{a^2+c_k^2}\right)\|z^k-z^*\|^2.
                 \end{eqnarray}
As mentioned, this inequality is tight when the maximal monotone operator $T$ is defined as (\ref{opeT}).
Note that
$\min(\gamma,2\gamma-\gamma^2)\frac{c_k^2}{a^2+c_k^2} \to 1$ when $c_k \to \infty$ and $\gamma=1$. Moreover, we have
$$0<\min(\gamma,2\gamma-\gamma^2)<1 \;\hbox{and} \;\;0<\frac{c_k^2}{a^2+c_k^2}\le1 $$
when $\gamma \in (0,2)$ and $\gamma \ne 1$.
Thus, if $c_k \to \infty$, the coefficient in (\ref{tmp2}) goes to $0$ only when $\gamma=1$. This excludes the hope of establishing the superlinear convergence for the exact version of the generalized PPA (\ref{RPPAscheme-ck}) with $\gamma \ne 1$ even when $c_k \to \infty$.

\section{Application to ALM}
\label{ALM}

%[{\bf Is it possible to consider a relaxed ALM and its linear convergence (both inexact and inexact versions?)}]

Previously, we have discussed the linear convergence rates for the generalized PPA schemes (\ref{RPPAscheme-ck}) and (\ref{RPPAscheme-ck-inexact}) in the generic setting of (\ref{maxsol}) where $T$ is an abstract maximal operator. In this and next sections, we specify our discussion to some special convex minimization models and discuss the linear convergence rates for two important algorithms which can be obtained by specifying the exact version of the generalized PPA scheme (\ref{RPPAscheme-ck}). For succinctness, discussions for their inexact counterparts stemming from the inexact version (\ref{RPPAscheme-ck-inexact}) are omitted.

Let us first recall some known results and summarize them in the following two lemmas. The proof of the first lemma can be found in
%\cite{BG77} and
\cite{Rock70}, and the proof of the second is trivial.
%, respectively.

%\begin{lemma}
%Let $\partial g$ be the subdiffernetial of a convex function $g: \Re^n \to \Re$. Then $\partial g$ is $\frac{1}{\beta}$-Lipschitz if and only if $\partial g$ is $\beta$-firmly nonexpansive.
%\end{lemma}

\begin{lemma}
\label{StrLip}
Let $f: \Re^n \to \Re$ be closed, proper and convex. Then, we have
%\begin{eqnarray}
%\label{cong}
%f^*(y):=\sup_x \{\langle x,y\rangle-f(x)\}.
%\end{eqnarray}
%Then, we have
\begin{itemize}
\item[i)] If $f$ is $\mu_f$-strongly convex, then $f^*$ is differentiable and $\nabla f^*$ is $(1/\mu_f)$-Lipschitz continuous.
\item[ii)] If $f$ is differentiable and $\nabla f$ is $L_f$-Lipschitz continuous, then $f^*$ is $1/L_f$-strongly convex.
 %[{\bf Check it, your previous version has typo}]
\end{itemize}
\end{lemma}

\begin{lemma}\label{stco}
Let $f: \Re^n \to \Re$ be closed, proper and strongly convex; let $\partial f$ be the subdifferential of $f$. Then, $\partial f$ is strongly monotone.
\end{lemma}

\subsection{Preliminaries of ALM}

First, we consider a canonical convex minimization model with linear constraints:
\begin{eqnarray}
\label{ConvexALM}
\begin{array}{cl}
\min& f(x)\\
s.t.    & Ax=b.
\end{array}
\end{eqnarray}
where $f: R^n \rightarrow (-\infty,\infty]$ is a closed and convex function, $A\in R^{m\times n}$ and $b \in R^m$.
For solving (\ref{ConvexALM}), a benchmark is the augmented Lagrangian method (ALM) originally proposed in \cite{Hes,Powell}. Its iterative scheme reads as
\begin{eqnarray}
\left\{\begin{array}{cl}
x^{k+1}&=\arg\min_x \{f(x)-\langle p^k, Ax\rangle+\frac{1}{2}c_k \|Ax-b\|^2\},\\
p^{k+1}&=p^k-c_k (Ax^{k+1}- b),
\end{array}
\right. \label{ALMS}
\end{eqnarray}
where $p^k$ is the Lagrange multiplier and $c_k>0$ is the penalty parameter of the linear constraints. As analyzed in \cite{Rock76}, the dual problem of (\ref{ConvexALM}) is
\begin{eqnarray}
\label{DConPALM}
\max_p \; \{-f^*(A^\top p)+\langle b, p\rangle \},
\end{eqnarray}
where ``*" denotes the conjugate of a function, see \cite{Rock76}. Thus, solving (\ref{DConPALM}) is equivalent to
\begin{eqnarray}
\label{Dule}
0\in S_{{\cal A}}(p):=A\cdot \partial f^*\cdot (A^\top p)-b,
\end{eqnarray}
which is a specific application of the generic setting of (\ref{maxsol}) with $T=S_{{\cal A}}$.
%where
%\begin{eqnarray}
%\label{operSA}S_{{\cal A}}(p)=A\cdot \partial f^*\cdot (A^\top p)-b.
%\end{eqnarray}
In \cite{Rock76MOR}, it was precisely analyzed that the ALM scheme (\ref{ALMS}) is an application of the PPA (\ref{PPA}) to the dual problem (\ref{Dule}). Also in \cite{Rock76}, it was mentioned that the mapping $S_{{\cal A}}(p)$ defined in (\ref{Dule}) is maximal monotone.
%[{\bf verify it}]

\subsection{A Generalized ALM}

Following the analysis in \cite{Rock76}, it is easy to see that if we apply the generalized PPA scheme (\ref{RPPAscheme-ck}) to (\ref{Dule}), we can obtain a generalized ALM scheme as follows
\begin{eqnarray}
\left\{\begin{array}{cl}
x^{k+1}&=\arg\min_x \{f(x)-\langle p^k, Ax\rangle+\frac{1}{2}c_k \|Ax-b\|^2\},\\
p^{k+1}&=p^k-\gamma c_k (Ax^{k+1}- b),
\end{array}
\right. \label{GALMS}
\end{eqnarray}
which differs from the original ALM (\ref{ALMS}) in that there is a parameter $\gamma\in(0,2)$ for updating the Lagrange multiplier $p^{k+1}$. The details are presented in the following theorem.

\begin{theorem}\label{GALM-exact}
The generalized ALM scheme (\ref{GALMS}) is an application of the exact version of the generalized PPA (\ref{RPPAscheme-ck}) to (\ref{Dule}).
\end{theorem}
\begin{proof}
First, the generalized ALM (\ref{GALMS}) can be rewritten as
\begin{eqnarray}
\left\{\begin{array}{cl}
x^{k+1}&=\arg\min_x \{f(x)-\langle p^k, Ax\rangle+\frac{1}{2}c_k \|Ax-b\|^2\},\\
\tilde p^{k}&=p^k- c_k (Ax^{k+1}- b),\\
p^{k+1}&=p^k-\gamma(p^k-\tilde p^k).
\end{array}
\right. \label{GALMSR}
\end{eqnarray}
Since the first-order optimality condition of the $x$-subproblem in (\ref{GALMSR}) is
$$A^\top(p^k-c_k(Ax^{k+1}-b))\in \partial f(x^{k+1}),$$
it follows from the second equation in (\ref{GALMSR}) that $A^\top \tilde p^{k}\in \partial f(x^{k+1})$. Then, we have
$$ Ax^{k+1}-b\in A\cdot \partial f^*\cdot (A^\top \tilde p^k)-b.$$
We thus conclude that
$$\tilde p^k=p^k-c_k(Ax^{k+1}-b)=p^k-c_k( A\cdot \partial f^*\cdot (A^\top \tilde p^k)-b)=p^k-c_k S_{{\cal A}}(\tilde p^k),$$
which implies that $\tilde p^k=J_{c_kS_{\cal A}}(p^k)$. Then, it follows from the last equation in (\ref{GALMSR}) that
$$p^{k+1}=p^k-\gamma (p^k-J_{c_kS_{\cal A}}(p^k)),$$
meaning that the generalized ALM scheme (\ref{GALMS}) is an application of (\ref{RPPAscheme-ck}) to (\ref{Dule}). The proof is complete.
\end{proof}

\subsection{Linear Convergence of ALM schemes}

Below, we show some conditions that can sufficiently ensure that the mapping  $S_{{\cal A}}^{-1}$ ($S_{{\cal A}}$ defined in (\ref{Dule})) is Lipschitz continuous at $0$ with positive modulus, and thus guarantee the linear convergence rate of the generalized ALM (\ref{GALMS}) (also the original ALM (\ref{ALMS}) if taking $\gamma=1$ in (\ref{GALMS})).

\begin{theorem}
\label{stronglipALM}
Let $S_{{\cal A}}$ be defined in (\ref{Dule}) and $\{p^k\}$ be the sequence generated by the generalized ALM scheme (\ref{GALMS}). For the model (\ref{ConvexALM}), if $f$ is convex and differentiable, $\nabla f$ is $L_f$-Lipschitz continuous, and the matrix $A$ is full row rank. Then, we have
\begin{itemize}
\item[(1)] The mapping $S_{{\cal A}}$ is strongly monotone.
\item[(2)] The mapping  $S_{{\cal A}}^{-1}$ exists and it is Lipschitz continuous at $0$ with positive modulus.
\item[(3)] The sequence $\{p^k\}$  converges linearly to a zero point of $S_{{\cal A}}$.
\end{itemize}
\end{theorem}

\begin{proof} (1) Since $f$ is differentiable and $\nabla f$ is $L_f$-Lipschitz continuous, it follows from Property (ii) of Lemma \ref{StrLip} that $f^*$ is $\frac{1}{L_f}$-strongly convex.
Then, it follows from Lemma \ref{stco} that $\partial f^*$ is $\frac{1}{L_f}$-strongly monotone. For any $p,p' \in R^m$; $w\in  S_{{\cal A}} (p)$ and $w'\in S_{{\cal A}} (p')$; there exist $u\in\partial f^*\cdot (A^\top p)$ and $u'\in\partial f^*\cdot (A^\top p')$ such that $w=Au$ and $w'=Au'$. We thus have
$$
 \langle w- w',p-p'  \rangle =\langle A ( u-u'),p-p'\rangle =\langle u-u',A^\top (p - p')\rangle
 \ge\frac{1}{L_f}\|A^\top (p-p')\|^2 \ge\frac{1}{L_f}\lambda_{\min}({AA^\top})\|p-p'\|^2,
 $$
  %\begin{eqnarray*}
% \begin{array}{rl}
% \langle w- w',p-p'  \rangle &=\langle A ( u-u'),p-p'\rangle\\
% &=\langle u-u',A^\top (p - p')\rangle\\
% %&= \langle  A\cdot \partial f^*\cdot (A^\top p)- A\cdot \partial f^*\cdot (A^\top p'),p-p'\rangle \\
% %&= \langle  \partial f^*\cdot (A^\top p)- \partial f^*\cdot (A^\top p'),A^\top p-A^\top p'\rangle
%%\\
% &\ge\frac{1}{L_f}\|A^\top (p-p')\|^2\\
% &\ge\frac{1}{L_f}\lambda_{\min}({AA^\top})\|p-p'\|^2,
% \end{array}
% \end{eqnarray*}
where the first inequality is because of the $\frac{1}{L_f}$-strongly convex of $\partial f^*$, and $\lambda_{\min}({AA^\top })$ is the minimal eigenvalue of $AA^\top $ with $\lambda_{\min}({AA^\top })>0$ because $A$ is assumed to be full row rank. Thus, it follows from Definition \ref{coer} that the mapping $S_{{\cal A}}$ is strongly monotone.

(2) This is obvious based on Definitions \ref{coer} and \ref{invLipze}. (3) This is just a conclusion of Theorem \ref{rel_PR} with $T=S_{{\cal A}}$ and the second assertion. The proof is complete.
%Now, we prove (2). Let $\{(x^k,p^k)\}$ be generated by the generalized ALM (\ref{GALMS}). According to Theorem \ref{GALM-exact}, the sequence $p^k$ can be written as
%$$p^{k}=p^{k-1}-\gamma (p^{k-1}-J_{c_kS_{\cal A}}(p^{k-1})).$$
%Therefore, applying Theorem \ref{rel_PR} with $T=S_{{\cal A}}$, we have that the sequence $\{p^k\}$ converges linearly when the mapping $S_{{\cal A}}^{-1}$ is Lipschitz continuous at zero  with positive modulus. The proof is complete.
$\hfill$
%For the inexact schemes (\ref{GinALMS}) and (\ref{GinALMSII}), we can derive the linear convergence of $\{p^k\}$  by applying
%Theorem \ref{rel PRinexact} with $T=S_{\cal A}$ and Lemma \ref{inexALMIvsII} under the same assumption.
\end{proof}

\section{Application to ADMM}
\label{ADMMChar}
In this section, we consider another convex minimization model with a separable objective function:
\begin{eqnarray}
\label{ConP}
\begin{array}{cl}
\min_{x}& \{f(x)+g(Mx)\}
\end{array}
\end{eqnarray}
where $f: R^n \rightarrow (-\infty,\infty]$ and $g: R^m\rightarrow (-\infty,\infty]$ are closed and convex functions, and the matrix $M\in R^{m\times n}$.
Again, we only focus on the specification of the exact version of the generalized PPA (\ref{RPPAscheme-ck}) and discuss how to ensure its linear convergence rate for the particular convex minimization context (\ref{ConP}).

\subsection{Preliminaries of ADMM} \label{preadmm}

One particular case of (\ref{ConP}) with a wide range of applications is where the functions $f$ and $g$ have their own properties and it is necessary to treat them individually in algorithmic design. For this purpose, we can reformulate (\ref{ConP}) as
\begin{eqnarray}
\label{GConP}
\begin{array}{cl}
\min & f(x)+g(w)\\
s.t.   & Mx=w,
\end{array}
\end{eqnarray}
where $w \in R^m$ is an auxiliary variable. For solving (\ref{GConP}), a benchmark is the ADMM scheme originally proposed in \cite{GM}. The ADMM scheme for (\ref{GConP}) reads as
\begin{eqnarray}
\left\{\begin{array}{cl}
x^{k+1}&=\arg\min_x \{f(x)+\langle p^k, Mx\rangle+\frac{1}{2}\lambda \|Mx-w^k\|^2\},\\
w^{k+1}&=\arg\min_w \{g(w)-\langle p^k, w\rangle +\frac{1}{2}\lambda \|Mx^{k+1}-w\|^2,\\
p^{k+1}&=p^k+\lambda (Mx^{k+1}- w^{k+1}),
\end{array}
\right. \label{ADMMS}
\end{eqnarray}
where $p^k$ is the Lagrange multiplier and $\lambda>0$ is a penalty parameter of the linear constraints in (\ref{GConP}).
%[{\bf Here we can only consider $\lambda$, not $\lambda_k$ as before? It is better to be consistent}]

Next, we recall some results in \cite{EcBe,Gabay83} to demonstrate that the ADMM is indeed a special case of the PPA (\ref{PPAscheme}). All the details can be found in \cite{EcBe}. First, the dual of (\ref{GConP}) is
\begin{eqnarray}
\label{DuaP}
\max_{p\in R^m} - (f^*(-M^\top p)+g^*(p))
\end{eqnarray}
where ``$f^*$" and ``$g^*$" denote the conjugate of the convex functions $f$ and $g$, respectively. Let
\begin{eqnarray}
\label{operABnew}
{\cal A}:=\partial [f^*\cdot (-M^\top)] \;\mbox{and} \;{\cal B}:=\partial g^*.
\end{eqnarray}
As shown in \cite{Rock70}, both ${\cal A}$ and ${\cal B}$ defined in (\ref{operABnew}) are maximal monotone. Then, (\ref{DuaP}) can be written as
\begin{eqnarray}
\label{ABsol}
0\in {\cal A}(p)+{\cal B}(p).
\end{eqnarray}
We use $J_{\lambda A}$ and $J_{\lambda B}$ to denote the resolvent operators of ${\cal A}$ and ${\cal B}$, respectively. Moreover, we denote
\begin{eqnarray}
\label{operG}
G_{\lambda,{\cal A},{\cal B}}=J_{\lambda {\cal A}}(2J_{\lambda {\cal B}} - I)+(I-J_{\lambda{\cal B}})
\end{eqnarray}
and
\begin{eqnarray}
\label{operS}
S_{\lambda,{\cal A},{\cal B}}:=G_{\lambda,{\cal A},{\cal B}}^{-1} - I.
\end{eqnarray}
As shown in \cite{EcBe}, $S_{\lambda,{\cal A},{\cal B}}$ is maximal monotone when ${\cal A}$ and ${\cal B}$ are both maximal monotone. Indeed, the definition of $S_{\lambda,{\cal A},{\cal B}}$ can be expressed as
\begin{equation}\label{operSC}
S_{\lambda,{\cal A},{\cal B}}=\{(v+\lambda b,u-v)|(u,b)\in{\cal B},(v,a)\in {\cal A}, v+\lambda a=u-\lambda b\},
\end{equation}
where ${\cal A}$ and ${\cal B}$ are defined in (\ref{operABnew}). Moreover, let $p^*$ be an solution point of (\ref{ABsol}) and $z^*$ a solution point of
\begin{eqnarray}\label{ADMMPPAP} 0\in S_{\lambda,{\cal A},{\cal B}}(z),\end{eqnarray}
and let the sequence $\{z^k\}$ be iteratively represented by
\begin{eqnarray}
\label{DRz}
z^{k+1}= J_{\lambda {\cal A}}\big ( (2J_{\lambda {\cal B}} - I)(z^k)\big)+(I-J_{\lambda{\cal B}})(z^k).
\end{eqnarray}
Indeed, (\ref{DRz}) is exactly the application of the Douglas-Rachford splitting method (DRSM) in \cite{DR56,LM} to (\ref{ABsol}). According to \cite{EcBe}, we know some conclusions such as: (1) If $z^*$ is a solution point of (\ref{ADMMPPAP}), then we have $p^*:=J_{\lambda {\cal B}}(z^*)$ is a solution point of (\ref{ABsol}); and (2) If $p^*$ is a solution point of (\ref{ABsol}) and $(x^*,w^*)$ is a solution point of (\ref{GConP}), then we have $x^*\in \partial f^*\cdot( -M^\top p^*)$ and $w^*\in \partial g^*(p^*)$.

Applying the scheme (\ref{RPPAscheme-ck}) with $T=S_{\lambda,{\cal A},{\cal B}}$, we obtain the exact version of the generalized PPA scheme
\begin{eqnarray}
\label{relaxPR}
z^{k+1}= z^k - \gamma (z^k - J_{S_{\lambda, {\cal A},{\cal B}}}(z^k))\;\hbox{with}\; \gamma\in(0,2).
\end{eqnarray}
%\begin{eqnarray}
%\label{relaxPR-inexact}
%z^{k+1} \approx z^k - \gamma (z^k - J_{S_{\lambda, {\cal A},{\cal B}}}(z^k))\;\; \gamma\in(0,2).
%\end{eqnarray}
Indeed, via (\ref{relaxPR}), the following exact version of the generalized ADMM scheme proposed in \cite{EcBe} can be recovered
\begin{eqnarray}
\left\{\begin{array}{cl}
x^{k+1}&=\arg\min_x \{f(x)+\langle p^k, Mx\rangle+\frac{1}{2}\lambda \|Mx-w^k\|^2\},\\
w^{k+1}&=\arg\min_w \{g(w)-\langle p^k, w\rangle +\frac{1}{2}\lambda \|\gamma Mx^{k+1}+(1-\gamma) w^k-w\|^2\},\\
p^{k+1}&=p^k+\lambda (\gamma Mx^{k+1}+(1-\gamma) w^k- w^{k+1}).
\end{array}
\right. \label{RADMMS}
\end{eqnarray}
In the following, we elucidate the relationship between the sequence $\{(x^k,w^k,p^k)\}$ generated by the generalized ADMM (\ref{RADMMS}) and $\{z^k\}$ represented by (\ref{relaxPR}); and demonstrate that the generalized ADMM (\ref{RADMMS}) can be written compactly as (\ref{relaxPR}). The
following lemma also clearly shows that the generalized ADMM (\ref{RADMMS}) is an application of the generalized PPA (\ref{RPPAscheme-ck}) with $T=S_{\lambda,{\cal A},{\cal B}}$ and $c_k\equiv1$ to (\ref{ADMMPPAP}).
%We first recall some known results to analyze the linear convergence rate for the generalized ADMM (\ref{RADMMS}).
%Note that the original ADMM (\ref{ADMMS}) is a special case
%of the generalized ADMM (\ref{RADMMS}) with $\gamma=1$.
%Thus, the conclusion for original ADMM follows directly from the generalized ADMM.

%\begin{lemma}
%Let the operators ${\cal A}$, ${\cal B}$ and $S_{\lambda,{\cal A},{\cal B}}$ are defined in (\ref{operABnew}) and (\ref{operS}), respectively. Let $\{(x^k,w^k,p^k)\}$ be generated by the ADMM (\ref{RADMMS}) and $\{q^k\}$ be given in (\ref{qk}). Then, the operator $S_{\lambda,{\cal A},{\cal B}}^{-1}$
%is characterized by
%\begin{eqnarray}
%\label{Sope_admm}
%S_{\lambda,{\cal A},{\cal B}}^{-1}=\{(p^k-q^k,q^k+\lambda w^k)| (p^k,w^k)\in{\cal B}, (q^k,-Mx^{k+1})\in {\cal A}, q^k+\lambda (-M x^{k+1})=p^k-\lambda w^k\},\nn\\
%\end{eqnarray}
%\end{lemma}
%\begin{proof}
%Based on the optimality condition of (\ref{RADMMSEQ}) and the definition of $q_k$ in (\ref{qk}),
%it holds that
%\begin{eqnarray}
%\label{opti}
%\begin{array}{rl}
%-Mx^{k+1}&\in {\cal A}(q^k)\\
%\tilde w^{k+1}&\in {\cal B}(\tilde p^{k+1}).
%\end{array}
%\end{eqnarray}
%Then, using (\ref{qk}) and the definition of $S_{\lambda,{\cal A},{\cal B}}$ in (\ref{operS}), the assertion is proved.
%\end{proof}
%\vspace{0.5cm}

\begin{lemma}
\label{relPRRADMM}
Let $\{(x^k,w^k,p^k)\}$ be generated by the generalized ADMM (\ref{RADMMS}) and $\{z^k\}$ be represented by (\ref{relaxPR}); the operator $S_{\lambda,{\cal A},{\cal B}}$ be defined in (\ref{operS}). Assume that the initial points satisfy with $z^0=p^0+\lambda w^0$  and $p^0=J_{\lambda \cal B}(z^0)$.
Then, it holds that $z^k =p^k+\lambda w^k$ and $p^k=J_{\lambda \cal B}(z^k)$ for all iterates.
\end{lemma}
%{{\color blue} To my opinion, it was not  verified in \cite{EcBe}.}

\begin{proof} The proof is mainly inspired by Theorem 8 in \cite{EcBe}. We provide the proof for completeness. First, we introduce an auxiliary sequence $\{q^k\}$  as
\begin{eqnarray}
\label{qk}
q^k=p^k+\lambda(M x^{k+1}-w^k).
\end{eqnarray}
Assume that $z^k=p^k+\lambda w^k$, in the following we show that $z^{k+1}=p^{k+1}+\lambda w^{k+1}$.
First, denote $\tilde z^k=J_{S_{\lambda,{\cal A},{\cal B}}}(z^k)$.
Since $J_{\lambda {\cal B}}(z^k)=p^k$, it yields  that $z^k=p^k+\lambda w^k$ and $w^k\in{\cal B}(p^k)$. Then, we have
$$
J_{\lambda {\cal A}}(2J_{\lambda {\cal B}}-I)(z^k)=J_{\lambda {\cal A}}(2p^k-z^k)=J_{\lambda {\cal A}}(p^k-\lambda w^k)=J_{\lambda {\cal A}}(q^k-Mx^{k+1})=q^k,
$$
where the second equality follows from $z^k=p^k+\lambda w^k$, the third follows from the definition of $q^k$ (\ref{qk}), and the last comes from $-Mx^{k+1}\in {\cal A}(q^k)$.
We thus have
\begin{eqnarray*}
\tilde z^{k}&=&J_{\lambda {\cal A}}(2J_{\lambda {\cal B}}-I)(z^k)+(I-J_{\lambda {\cal B}})(z^k)=q^k+z^k-p^k =q^k+\lambda w^k.
\end{eqnarray*}
Then, we have
\begin{eqnarray*}
z^{k+1}&=&z^k-\gamma(z^k-\tilde z^k) =p^k+\lambda w^k-\gamma(p^k-q^k)\nn\\
       &=&p^k+\lambda w^k-\gamma (-\lambda M x^{k+1}+\lambda w^k)\nn\\
       &=&p^k+\lambda(\gamma Mx^{k+1}+(1-\gamma)w^k)\nn\\
       &=&p^{k+1}+\lambda w^{k+1},
\end{eqnarray*}
where the second equality follows from the fact $z^k=p^k+\lambda w^k$, the third is because of the definition of $q^k$ (\ref{qk}) and the last comes from the update scheme of $p^{k+1}$ in (\ref{RADMMS}). The proof is complete. $\hfill$
\end{proof}

Finally, let us first present a lemma; its proof can be found in \cite{EcBe}.

\begin{lemma}\label{operGPR}
The operator $G_{\lambda,{\cal A}, {\cal B}}$ defined in (\ref{operG}) is firmly nonexpansive and it satisfies
%[{\bf If $G_{\lambda, {\cal A},{\cal B}}$ is not single-values, the following writing is not good, you need to change it as what I write for Theorem 3.2}]
\begin{eqnarray}
\label{contrGoper}
 \langle G_{\lambda,{\cal A}, {\cal B}}(z)-G_{\lambda,{\cal A}, {\cal B}}(z'), z-z'\rangle &\ge& \|G_{\lambda,{\cal A}, {\cal B}}(z)-
G_{\lambda,{\cal A}, {\cal B}}(z')\|^2\nn\\
&&+\langle (I-J_{\lambda\cal B})(z)- (I - J_{\lambda \cal B}(z'), J_{\lambda \cal B}(z)-J_{\lambda \cal B}(z'))\rangle, \;\forall \; z,z'\in H.
\end{eqnarray}
\end{lemma}

\subsection{When Does the Assumption Hold?}

Based on our previous analysis, it is clear that the linear convergence of the generalized ADMM (\ref{RADMMS}) can be ensured by the assumption ``The mapping $S_{\lambda,{\cal A},{\cal B}}^{-1}$ ($S_{\lambda,{\cal A},{\cal B}}$ defined in (\ref{operS})) exists and it is Lipschitz continuous at $0$ with positive modulus". When the specific model (\ref{ConP}) is considered, it is interesting to discern sufficient conditions that can ensure this assumption and thus guarantee the linear convergence of the generalized ADMM scheme (\ref{RADMMS}); this is the main purpose of this subsection. We also refer to, e.g., \cite{Boley,DengYin,HY-SINUM} for discussions on the linear convergence of the original ADMM (\ref{ADMMS}) for some special cases.

In the following, we show one scenario that can sufficiently ensure the mentioned assumption for the specific model (\ref{ConP}) and thus guarantee the linear convergence of the the sequence $\{z^k\}$ represented by (\ref{relaxPR}).

\begin{theorem} \label{stronglip}
For the model (\ref{ConP}), if the function $g$ is differentiable and strongly convex, and $\nabla g$ is Lipschitz continuous near a solution point,
then we have
\begin{itemize}
\item[(1)]  The operator ${\cal B}:=\nabla g^*$ is both strongly monotone and Lipschitz continuous near the solution point.
%\item the relaxed Assumption A holds when the operator ${\cal B}:=\partial g^*$ is both coercive and Lipschitz near the solution;
\item[(2)]  The mapping $S_{\lambda,{\cal A},{\cal B}}^{-1}$ is Lipschitz continuous at the iterate $\tilde z^k:=J_{S_{\lambda,{\cal A},{\cal B}}}(z^k)$ with positive modulus when $k$ stays large enough.
\item[(3)] The sequence $\{z^k\}$ represented by (\ref{relaxPR}) converges linearly to a solution point of (\ref{ADMMPPAP}).
\end{itemize}
\end{theorem}
\begin{proof}
(1) According to Lemma \ref{StrLip}, we know that $g^*$ is differentiable. Thus,
${\cal B}:=\nabla g^*$ is both strongly monotone and Lipschitz continuous near the solution point of (\ref{DuaP})
 %[{\bf ``solution point" of which model? should be specified}]
 according to Lemmas \ref{StrLip} and \ref{stco}.
Thus, the first conclusion is proved.

(2) Next, we show that the Lipschitz constant of the operator $G_{\lambda,{\cal A}, {\cal B}}$ is less than 1.
Note that $\frac{1}{\lambda}((I-J_{\lambda \cal B})(z))\in {\cal B}(J_{\lambda \cal B}(z))$. Let us assume that the strongly monotone modulus of ${\cal B}$ is $\alpha$. That is,
\begin{eqnarray*}
\langle    (I-J_{\lambda \cal B})(z) -(I-J_{\lambda \cal B})(z'),J_{\lambda \cal B}(z)-J_{\lambda \cal B}(z')\rangle\ge \lambda\alpha \|J_{\lambda \cal B}(z)-J_{\lambda \cal B}(z')\|^2, \;\;\forall z,z' \in H.
\end{eqnarray*}
Moreover, let us assume that Lipschitz continuous constant of ${\cal B}$ is $\beta$. Then, we have
\begin{eqnarray*}
\label{}\|z-z'\|^2=\|J_{\lambda \cal B}(z)-J_{\lambda \cal B}(z')+\lambda{\cal B}(J_{\lambda \cal B}(z))-\lambda{\cal B}(J_{\lambda \cal B}(z'))\|^2\le(1+\lambda\beta)^2\|J_{\lambda \cal B}(z)-J_{\lambda \cal B}(z')\|^2, \;\;\forall z,z' \in H.
\end{eqnarray*}
Combining these two inequalities, we get
\begin{eqnarray}
\label{Bcoli1}
\langle    (I-J_{\lambda \cal B})(z) -(I-J_{\lambda \cal B})(z'),J_{\lambda \cal B}(z)-J_{\lambda \cal B}(z')\rangle\ge \frac{\lambda \alpha}{(1+\lambda\beta)^2}\|z-z'\|^2, \;\;\forall z,z' \in H.
\end{eqnarray}
Then, it follows from Lemma \ref{operGPR} and (\ref{Bcoli1}) that
 %[{\bf If $G_{\lambda, {\cal A},{\cal B}}$ is not single-values, the following writing is not good, you need to change it as what I write for Theorem 3.2. Several places below.}]
\begin{eqnarray*}
\|z-z'\|^2\ &\ge& \langle G_{\lambda,{\cal A}, {\cal B}}(z)-G_{\lambda,{\cal A}, {\cal B}}(z'), z-z'\rangle \nn \\
&\ge& \|G_{\lambda,{\cal A}, {\cal B}}(z)-G_{\lambda,{\cal A}, {\cal B}}(z')\|^2+\langle (I-J_{\lambda\cal B})(z)- (I - J_{\lambda \cal B}(z'), J_{\lambda \cal B}(z)-J_{\lambda \cal B}(z'))\rangle\nn\\
&\ge& \|G_{\lambda,{\cal A}, {\cal B}}(z)-G_{\lambda,{\cal A}, {\cal B}}(z')\|^2 +\frac{\lambda \alpha}{(1+\lambda\beta)^2}\|z-z'\|^2, \;\;\forall z,z' \in H,
\end{eqnarray*}
where the first inequality follows from the non-expansiveness of the operator $ G_{\lambda,{\cal A}, {\cal B}}$; the second inequality is because of (\ref{contrGoper}) and the last inequality holds because of (\ref{Bcoli1}).
Consequently, we prove that
$$\|G_{\lambda,{\cal A}, {\cal B}}(z)-
G_{\lambda,{\cal A}, {\cal B}}(z')\|\le \sqrt{1-\frac{\lambda \alpha}{(1+\lambda\beta)^2}}\|z-z'\|,\;\;\forall z,z' \in H.
$$
Recall the definitions of the strongly monotonicity and the Lipschitz continuity of  ${\cal B}$. We have $\alpha \le \beta$ and thus the above inequality means the fact that the Lipschitz continuity constant of the operator $G_{\lambda,{\cal A}, {\cal B}}$ is less than 1.
Finally, it follows from Corollary \ref{rela} with $T=S_{\lambda,{\cal A},{\cal B}}$, $c_k\equiv1$ and  $G=G_{\lambda,{\cal A},{\cal B}}$ that
the mapping $S_{\lambda,{\cal A},{\cal B}}^{-1}$ is Lipschitz continuous at the iterate $\{\tilde z^k\}$ with  positive modulus when $k$ stays large enough, where $\tilde z^k:=J_{S_{\lambda,{\cal A},{\cal B}}}(z^k)$.

(3) Finally, the linear convergence of the sequence $\{z^k\}$ follows assertion (2) of Theorem \ref{rela} with $c_k\equiv1$ and assertion (2) immediately. The proof is complete. $\hfill$
\end{proof}

Note that the linear convergence of $\{z^k\}$ represented by (\ref{relaxPR}) can be easily specified as the linear convergence of the generalized ADMM scheme (\ref{RADMMS}) in terms of the variables in (\ref{GConP}) and its dual. We summarize the specifications in the following corollary and omit the proofs.

\begin{corollary} \label{Theor2}
When the sequence $\{z^k\}$ represented by (\ref{relaxPR}) converges linearly to a solution point of (\ref{ADMMPPAP}), we have
\begin{itemize}
%\item [(1)] The sequence $\{z^k\}$ converges Q-linearly to a solution point $z^*$ of the problem (\ref{ADMMPPAP}).
\item [(1)] The sequence $\{p^k\}$ converges R-linearly to a solution point $p^*$ of the dual problem (\ref{DuaP}).
\item [(2)] The sequence $\{w^k\}$ converges R-linearly to a solution point $w^*$ of the primal problem (\ref{GConP}).
%\item [(4)] The sequence $\{q^k\}$ converges R-linearly to $p^*$.
\item [(3)] The sequence $\{Mx^{k}\}$ converges R-linearly to $Mx^*$, where $x^*$ is a solution point of the primal problem (\ref{ConP}). Moreover, if $M$ is full column rank, then the sequence $\{x^k\}$ converges R-linearly to $x^*$, where $x^*=(M^\top M)^{-1} M^\top  (M x^*)$.
\end{itemize}
\end{corollary}

\begin{remark}
Under one of the following conditions, we can also prove the conclusion ``The mapping $S_{\lambda,{\cal A},{\cal B}}^{-1}$ ($S_{\lambda,{\cal A},{\cal B}}$ defined in (\ref{operS})) exists and it is Lipschitz continuous at $0$ with positive modulus". We omit the proof because it is analogous to that of Theorem \ref{stronglip}.
\begin{itemize}
%\item The matrix $M$ is full column rank, the function $g$ is strongly convex and $\nabla g$ is Lipschitz continuous near the solution.
\item[(1)] The matrix $M$ is full row rank, the function $f$ is strongly convex and $\nabla f$ is Lipschitz continuous near $x^*$, where $x^*$ is a solution point of (\ref{ConP}).
\item[(2)] The matrix $M$ is full row rank, the function $f$ is  convex and $g$ is strongly convex near $Mx^*$,
 and $\nabla f$ is Lipschitz continuous near $x^*$, where $x^*$ is a solution point of (\ref{ConP}).
\item[(3)]  The matrix $M$ is full rank, the function $f$ is strongly convex near $x^*$ and $g$ is convex, and  $\nabla g$  are Lipschitz continuous near $Mx^*$, where $x^*$ is a solution point of (\ref{ConP}).
\end{itemize}
Together with the condition in Theorem \ref{stronglip}, these conditions coincide with the conditions in \cite{DYin2} (when $B=-I$ and $b=0$ in the model (2) therein) to ensure the linear convergence of the generalized ADMM (\ref{RADMMS}) for solving (\ref{ConP}). In other words, the assumption ``The mapping $S_{\lambda,{\cal A},{\cal B}}^{-1}$ ($S_{\lambda,{\cal A},{\cal B}}$ defined in (\ref{operS})) exists and it is Lipschitz continuous at $0$ with positive modulus" is weaker than these conditions.
\end{remark}

\begin{remark}
In \cite{CYuan}, the linear convergence of the generalized ADMM (\ref{RADMMS}) for solving (\ref{ConP}) is ensured under the following assumptions:
(1). $M$ is full rank, $f$ is convex and differentiable, $\nabla f$ is Lipschitz continuous, and $g$ is strongly convex; (2). $f$ is strongly convex, $g$ is convex and differentiable, and $\nabla g$ is Lipschitz continuous. We here give some less strengthen conditions.
\end{remark}

\section{Conclusion}\label{conclusion}

In this paper, we extend the condition in \cite{Rock76} that can ensure the linear convergence of the proximal point algorithm (PPA) to a generalized PPA scheme. Both the exact and inexact versions of the generalized PPA are studied, and their linear convergence rates are established under the same condition as the original PPA in \cite{Rock76}. We specifically consider two convex optimization models and study the linear convergence rates for generalized versions of the benchmark augmented Lagrangian method (ALM) and the alternating direction method of multipliers (ADMM), both are special cases of the proposed generalized PPA. Some concrete conditions are specified in the convex optimization contexts. It is interesting to find that the condition in \cite{Rock76} turns out to be still weaker than most of the existing conditions in the literature that were proposed to ensure the linear convergence for various specific forms of the PPA. This study provides a unified understanding of the linear convergence of a family of operator splitting methods which have found a board spectrum of applications in various areas. These methods include the mentioned ALM, ADMM, their generalized and inexact versions, the Douglas-Rachford splitting method, the Peaceman-Rachford splitting method, and their generalized versions.


\begin{thebibliography}{40}


\bibitem{AusTeb05} A. Auslender and M. Teboulle, {\it Interior projection-like methods for monotone variational inequalities}, Math. Program., 104(2005), pp. 39-68.

\bibitem{ATB99} A. Auslender, M. Teboulle and S. Ben-Tiba, {\it A
      logarithmic-quadratic proximal method for variational inequalities},
      Comput. Optim. Appl., 12(1999), pp. 31-40.

\bibitem{Ber82} D. P. Bertsekas, Constrained Optimization and Lagrange
               Multiplier Methods, Academic Press, Newy York, 1982.

\bibitem{BO} E. Blum and W. Oettli, Mathematische Optimierung {\it Grundlagen und Verfahren.
           \"Okonometrie und Unternehmensforschung}, Springer-Verlag, Berlin-Heidelberg-New York, 1975.

\bibitem{Boley} D. Boley, {\it Local linear convergence of ADMM on quadratic or linear programs}, SIAM J. Optim, 23(2013), pp. 2183-2207.

\bibitem{BQ} J. V. Burke and M. J. Qian, {\it A variable metric proximal point algorithm for monotone operators}, SIAM J. Cont. Optim., 37(1998), pp. 353-375.

\bibitem{CGHY} X. J. Cai, G. Gu, B. S. He  and  X. M. Yuan,
{\it A relaxed customized proximal point algorithm for separable convex programming}, Sci. China Math., 56(2013), pp. 2179-2186.

\bibitem{CYuan} E. Corman and X. M. Yuan, {\it A generalized proximal point algorithm and its convergence rate},
SIAM Journal on Optimization, 24(2014), pp. 1614-1638.

\bibitem{DYin2} D. Davis and W. Yin,
{\it Faster convergence rates of relaxed Peaceman-Rachford and ADMM under regularity assumptions},
http://arxiv.org/pdf/1407.5210.pdf.


\bibitem{DengYin} W. Deng and W. Yin, {\it On the global and linear convergence of the generalized alternating direction method of multipliers}, Rice CAAM technical report 12-14, 2012.


\bibitem{DR56} J. Douglas and H. H. Rachford, {\it On the numerical solution of the
               heat conduction problem in 2 and 3 space variables}, Trans. Amer. Math. Soc., 82(1956), pp. 421-439.

\bibitem{EcBe} J. Eckstein and D. P. Bertsekas, {\it On the Douglas-Rachford splitting method
and the  proximal points algorithm for maximal monotone operators},
Math. Program., 55(1992), pp. 293-318.


\bibitem{FLHY} E. X. Fang, H. Liu, B. S. He and X. M. Yuan, {\it The generalized alternating direction method of multipliers: New theoretical insights and applications}, Math. Program. Comput, 7(2) (2015), pp. 149-187.

\bibitem{Fuk} M. Fukushima and H. Mine, {\it A generalized proximal point algorithm for certain nonconvex minimization problems}, Intern. J. Sys. Sci. 12(1981), pp. 989-1000.

\bibitem{Gabay83} D. Gabay, {\it Applications of the method of multipliers
                to variational inequalities},  Augmented Lagrange
                Methods: Applications to the Solution of Boundary-valued
                Problems, M. Fortin and R. Glowinski,  eds., North
                Holland, Amsterdam, The Netherlands, 1983, pp. 299-331.


\bibitem{GM} R. Glowinski and A. Marrocco, {\it Approximation par $\acute{e}$l$\acute{e}$ments finis d'ordre un et r$\acute{e}$solution par
p$\acute{e}$nalisation-dualit$\acute{e}$ d'une classe de probl$\grave{e}$mes non lin$\acute{e}$aires}, R.A.I.R.O., R2, 1975, pp. 41-76.

\bibitem{Gol79} E. G. Gol'shtein and N. V. Tret'yakov, {\it Modified Lagrangian in convex programming and their generalizations}, Math. Program. Study, 10(1979), pp. 86-97.

\bibitem{Guler91} O. G\"uler, {\it On the convergence of the proximal point algorithm
for convex minimization}, SIAM J. Optim., 1(1991), pp. 403-419.

\bibitem{Guler92} O. G\"uler, {\it New proximal point algorithms for convex minimization}, SIAM J. Optim., 2(1992), pp. 649-664.

\bibitem{HY-SINUM} D. R. Han and X. M. Yuan, {\it Local linear convergence of the alternating direction method of multipliers for quadratic programs}, SIAM J. Numer. Anal., 51(2013), pp. 3446-3457.

\bibitem{Hes} M. R. Hestenes, {\it Multiplier and gradient methods}, J. Optim. Theory Appli., 4(1969), pp. 303-320.


\bibitem{Kor} G.M. Korpelevich, {\it The extragradient method for finding saddle points and other problems}, Ekonomika i
Matematchskie Metody, 12(1976), pp. 747-756.

\bibitem{LM} P. L. Lions, B. Mercier, {\it splitting algorithms for the sum of two nonlinear operators},
SIAM J. Numer. Anal., 16(1979), pp. 964-979.

\bibitem{Mar70} B. Martinet, {\it Regularization d'inequations
variationelles par approximations successives}, Revue
Francaise d'Informatique et de Recherche Op\'erationelle, 4(1970), pp. 154-159.

\bibitem{Mar72} B. Martinet, Determination approchd\'ee d'un point fixe d'une application pseudo-contractante, C.R.
Acad. Sci. Paris, 274(1972), pp. 163-165.

\bibitem{Moreau} J. J. Moreau, {\it Proximit\'e et dualit
'e dans un espace Hilbertien}, Bull. Soc. Math. France, 93(1965), pp.
273-299.

\bibitem{Nemirovski2005} A. Nemirovski, {\it Prox-method with rate of
         convergence $O(1/t)$ for variational inequality with Lipschitz
         continuous monotone operators and smooth convex-concave saddle
         point problems}, SIAM J. Optim. 15(2005), pp. 229-251.

\bibitem{Nesterov} Y. E. Nesterov, {\it A method for solving the convex programming problem with convergence rate $O(1/{k^2})$}, Dokl. Akad. Nauk SSSR, 269(1983), pp. 543-547.

\bibitem{PR} D. H. Peaceman and H. H. Rachford, {\it The numerical solution of parabolic
elliptic differential equations}, J. Soc. Indust, Appl. Math., 3(1955), pp. 28-41.

\bibitem{Polyak} B. T. Polyak, {\it Introduction to Optimization, Translations Series in Mathematics and Engineering, Optimization
Software}, Publications Division, New York, 1987.

\bibitem{Powell} M. J. D. Powell, {\it A method for nonlinear constraints in minimization problems}, In Optimization edited by R. Fletcher, pp. 283-298, Academic Press,  New York, 1969.

\bibitem{Rock70} R. T. Rockafellar, {\it Convex analysis}, Princeton University Press, Princeton, N.J., 1970.

\bibitem{Rock76} R. T. Rockafellar, {\it Monotone  operators and the proximal point algorithm},
SIAM J. Con. Optimi., 14(1976), pp. 97-116.

\bibitem{Rock76MOR} R. T. Rockafellar, {\it Augmented Lagrangians and applications of the proximal point algorithm in convex programming},
Math. Oper. Res., 1(1976), pp. 877-898.

\bibitem{ST} R. Shefi and M. Teboulle, {\it Rate of convergence analysis of decmposition methods based on the proximal method of multipliers for convex minimization}, SIAM J. Optim. 24 (2014), pp. 269-297 .

\bibitem{Teb} M. Teboulle, {\it Convergence of proximal-like algorithms}, SIAM J. Optim., 7(1997), pp. 1069-1083.


\end{thebibliography}
\end{document}